\documentclass[reqno,11pt]{amsart}
\usepackage[all]{xy}

\usepackage{amsthm,amsfonts, amssymb, amscd, mathrsfs}
\usepackage{amsthm,amsfonts, amssymb, amscd}
\usepackage{euscript}

\usepackage{tikz}
\usetikzlibrary{matrix,arrows}
\usetikzlibrary{positioning}

\usepackage[normalem]{ulem}

\usepackage{hyperref}
\hypersetup{%
    linktoc=page
}

\let\oldtocsection=\tocsection
\let\oldtocsubsection=\tocsubsection
\renewcommand{\tocsection}[2]{\hspace{0em}\oldtocsection{#1}{#2}}
\renewcommand{\tocsubsection}[2]{\hspace{1em}\oldtocsubsection{#1}{#2}}
\newtheorem{Thm}{Theorem}[section]
\newtheorem{Lem}[Thm]{Lemma}
\newtheorem{Cor}[Thm]{Corollary}
\newtheorem{Prop}[Thm]{Proposition}

\theoremstyle{remark}
\newtheorem{rem}[Thm]{Remark}
\theoremstyle{remark}
\newtheorem{example}[Thm]{Example}
\theoremstyle{definition}

\numberwithin{equation}{section}
\newcommand{\G}{\mathbb{ G}}
\newcommand{\R}{\mathbb{ R}}           % Use for real numbers.
\newcommand{\C}{\mathbb{C}}           % Use for complex numbers.
\newcommand{\Z}{\mathbb{ Z}}           % Use for integers.
\newcommand{\F}{\mathbb{F}}           % Use for a field
\newcommand{\Q}{\mathbb{ H}}           % Use for quaternions
\newcommand{\V}{\mathbb{V}}  
\newcommand{\ad}{\operatorname{ad}}             % Use for adjoint action

\newcommand{\Hom}{\operatorname{Hom}}

\newcommand{\Spec}{\operatorname{Spec}}

\newcommand{\Ind}{\operatorname{Ind}}
\newcommand{\Res}{\operatorname{Res}}
\newcommand{\CInd}{\Eul{I}}

\renewcommand{\gcd}{\operatorname{gcd}}

\newcommand{\fb}{{\mathfrak b}}

\newcommand{\fg}{{\mathfrak g}}

\newcommand{\fl}{{\mathfrak l}}

\newcommand{\fp}{{\mathfrak p}}

\newcommand{\ft}{{\mathfrak t}}
\newcommand{\fu}{{\mathfrak u}}

\newcommand{\ga}{\alpha}

\newcommand{\gre}{\epsilon}

\newcommand{\gl}{\lambda}

\newcommand{\Eul}{\EuScript}

\newcommand{\cf}{\mathcal{F}}

 \newcommand{\cl}{\mathcal{L}}
 \newcommand{\cm}{\mathcal{M}}
 \newcommand{\cn}{\mathcal{N}}
 \newcommand{\co}{\mathcal{O}}
 \newcommand{\cp}{\mathcal{P}}
 \newcommand{\cq}{\mathcal{Q}}

\renewcommand{\tilde}{\widetilde}

\def\eeq{\end{equation}}

\def\label{\label}

\def\g{{\mathfrak g}}
\def\u{{\mathfrak u}}

\def\t{{\mathfrak t}}

\def\t{{\mathfrak t}}

\def\C{{\mathbb C}}

\def\R{{\mathbb R}}

\def\P{{\mathcal P}}
\def\Q{{\mathcal Q}}
\def\F{{\mathcal F}}

\def\W{W}

\def\Z{{\mathbb Z}}
\def\ad{\mbox{ad }}

\def\k{\kappa}

\newcommand{\toric}{\mathcal{V}}
\newcommand{\nilp}{\nu}
\newcommand{\tnilp}{\tilde{\nu}}
\newcommand{\mnmap}{\tilde{\eta}}
\newcommand{\mmap}{\tilde{\mu}}
\newcommand{\tormap}{\pi}

\newcommand{\Exterior}{\mathchoice{{\textstyle\bigwedge}}%
    {{\bigwedge}}%
    {{\textstyle\wedge}}%
    {{\scriptstyle\wedge}}}

\begin{document}
\parskip=4pt
\baselineskip=14pt

%%%%%%%%%%%%%%%%%%%%%%%%%%%%%%%%%%%%%%%%%%%%%%%%%%%%
%%%%%%%%%%%%%%%%%%%%%%%%%%%%%%%%%%%%%%%%%%%%%%%%%%%%
%%%%%%%%%%%%%%%%%%%%%%%%%%%%%%%%\end{center}%%%%%%%%%%%%%%%%%%%%
\title[A generalization of the Springer resolution]{Toric varieties and a generalization of the Springer resolution}

\author{William Graham}
\address{
Department of Mathematics,
University of Georgia,
Boyd Graduate Studies Research Center,
Athens, GA 30602
}
\email{wag@uga.edu}

\date{\today}

\vskip.2in

\begin{abstract}
Let $\fg$ be a semisimple Lie algebra and $\cn$ the nilpotent cone in $\fg$.  The
Springer resolution of $\cn$ has played an important role in representation theory.
The variety $\cn$ is equal to $\Spec R(\co^{pr})$, where $\co^{pr}$ is the open nilpotent orbit
in $\cn$.
This paper constructs and studies an analogue of the Springer resolution for the variety
$\cm = \Spec R(\widetilde{\co}^{pr})$,  where $\widetilde{\co}^{pr}$ is the universal
cover of $\co^{pr}$.  The construction makes use
of the theory of toric varieties.  Using this construction, we provide new proofs
of results of Broer and of Graham about $\cm$.  Finally, we show that
the construction can be adapted to covers of an arbitrary nilpotent orbit in $\fg$.
\end{abstract}

\maketitle

%\setcounter{tocdepth}{2}
%\tableofcontents
\begin{center}
{\it Dedicated to William Fulton on the occasion of his 80th birthday} 
\end{center}

\section{Introduction} \label{s.intro}
Let $G$ be a semisimple simply connected algebraic group with Lie algebra $\fg$, let
$\cn$ be the nilpotent cone in $\fg$, and let $\widetilde{\cn}$ denote the cotangent
bundle of the flag variety of $G$.  There is a map $\mu: \widetilde{\cn} \to \cn$, called
the Springer resolution, which plays an important role in representation theory.
The purpose of this paper is to construct and study an analogous map $\widetilde{\cm} \to \cm$,
where $\cm$ is defined as follows.  The nilpotent cone $\cn$ has a dense
$G$-orbit $\co^{pr}$, called the principal or regular nilpotent orbit, and 
$\cn = \Spec R(\co^{pr})$, where $R(\co^{pr})$ denotes the ring of
regular functions on $\co^{pr}$.  The variety $\cm$ is defined to be $\Spec R(\widetilde{\co}^{pr})$,
where $\widetilde{\co}^{pr}$ is the universal cover of $\co^{pr}$.  There is a commutative
diagram
\begin{equation} \label{e.diagramintro}
\begin{CD}
\tilde{\cm} @>{\tilde{\eta}}>> \tilde{\cn} \\
@V{\tilde{\mu}}VV         @VV{\mu}V \\
\cm @>{\eta}>> \cn.
\end{CD}
\end{equation}
Because the construction of
$\widetilde{\cm}$ makes use of toric varieties, the rich theory of these varieties
can be used to obtain detailed information about $\widetilde{\cm}$ and 
about the
map $\tilde{\cm} \to \tilde{\cn} $.

The motivation for constructing $\widetilde{\cm} \to \cm$ is to explore what constructions involving the Springer resolution can
yield in the setting of
 $\cm$, which is in some sense richer than $\cn$.  One 
original motivation
for considering covers of orbits came from the theory of Dixmier algebras, which
are algebras related to the infinite-dimensional representations of $\fg$.  More recently,
work of Russell and Graham, Precup and Russell (\cite{Rus:12}, \cite{GrPrRu:19}) has uncovered a close connection
in type $A$ between the map $\widetilde{\cm} \to \cm$ and Lusztig's generalized
Springer correspondence, which played an important role in the development of character
sheaves.  In this paper, we apply the commutative diagram above to give new proofs of a result
of Broer (see \cite{Bro:98}) showing 
that $\cm$ is Gorenstein with rational singularities, and of a formula for the 
$G$-module decomposition of $R(\widetilde{\co})$ (see \cite{Gra:92}). 

The center $Z$ of $G$ acts trivially on $\tilde{\cn}$.  To define $\tilde{\cm}$, we use the
theory of toric varieties to modify the
construction of $\tilde{\cn}$ and obtain a variety where $Z$ acts faithfully.  The variety
$\cn$ is isomorphic to $G \times^B \fu$, where $B$ is a Borel subgroup of $G$ with unipotent
radical $U$, and $\fu$ is the Lie algebra of $U$.  Using toric varieties, we
construct a variety $\widetilde{\fu}$ on which $Z$ acts faithfully, and define
$\cm = G \times^B \widetilde{\fu}$.  Since $\widetilde{\fu}$ maps to $\fu$, we obtain the map
$\tilde{\eta}: \widetilde{\cm} = G \times^B \widetilde{\fu} \to \widetilde{\cn} = G \times^B \fu$.

The proof of Theorem 4.1 in
\cite{McG:89} contains a related construction where a variety of the form
$G \times^B \widetilde{V}$ maps to the closure of $\widetilde{\co}^{pr}$ in a representation
of $G$.  The variety $\widetilde{V}$ is described as the closure of a
$B$-orbit of a representation, and it is not analyzed in detail in \cite{McG:89}.  
In our setting, because toric
varieties are so well understood, we can obtain detailed results about
$\widetilde{\fu}$ and hence about the geometry
of $\tilde{\cm}$.  This understanding is used in the results
of this paper, as well as in the applications to the generalized Springer correspondence.
Although $\tilde{\cm} \to \cm$ is an analogue of the Springer resolution, it is not a resolution
of singularities since $\tilde{\cm}$ is not in general smooth; it is
locally the quotient of a smooth variety by a finite group (see Proposition \ref{prop.orbifold2}). 
A genuine resolution of singularities of $\cm$ can be constructed from $\tilde{\cm}$ using toric resolutions, and we use this in our proofs.  Although a $G$-equivariant resolution of singularities exists by general theory, the commutative
diagram above, as well as the detailed results about the map $\widetilde{\cm} \to 
\widetilde{\cn}$, are crucial to the applications.

The construction of this paper can be adapted to any $G$-equivariant
covering of the principal orbit, but for simplicity, we consider only
the universal cover.  Also, in the final section of the paper, we show that
the construction can be generalized to the setting of covers of
any nilpotent orbit, not only the principal orbit (note that \cite[Theorem 4.1]{McG:89}
also applies to any nilpotent orbit cover).
However, the construction is more difficult to study explicitly for general orbits,
because the theory of toric varieties is not available, and further work would be required to
obtain precise results about the geometry.

The contents of the paper are as follows.   Section \ref{s.preliminaries} contains notation and background
results.  Section \ref{s.toric} introduces and studies the toric varieties used in our main construction.
In Section \ref{s.utilde}, we construct $\widetilde{\cm}$ and the commutative diagram
\eqref{e.diagramintro}.  Although $\widetilde{\cm}$ is not in general smooth, we construct
a resolution $\widehat{\cm} \to \widetilde{\cm}$, which is useful in proofs.  Section \ref{s.canonical} studies
canonical and dualizing sheaves on $\widehat{\cm}$ and $\widetilde{\cm}$ and uses results about these sheaves
to give new proofs of the results of Broer and of Graham (see \cite{Bro:98} and \cite{Gra:92}) mentioned above.
The methods of this paper lead to a slightly different formula for the $G$-module decomposition
of $R(\widetilde{\co}^{pr})$ than in \cite{Gra:92}, but they agree because of a fact about Weyl group conjugacy
of certain weights (see Proposition \ref{p.conjugacy}).  In Section \ref{s.nonnormal} we apply the techniques
of this paper to show that the closure in $\cm$ of the $B$-orbit $B \cdot \tilde{\nilp}$ is not normal; here $\tilde{\nilp}$
is a certain element in the orbit cover $\widetilde{\co}^{pr}$.   Finally, in Section \ref{s.otherorbits}, we show how this
construction extends to other nilpotent orbits besides the principal orbit.

\bigskip

{\em Acknowledgments:} This paper is dedicated to William Fulton on the occasion of
his 80th birthday.   The ideas and instruction he provided while I was a postdoc continue
to deeply influence my mathematical work, and I remain very grateful for his help.
I would also like to thank David Vogan, who initiated me into the study of nilpotent
orbits and their covers.
 Finally, I thank 
my collaborators Martha Precup and
Amber Russell for our joint work uncovering the connection of $\tilde{\cm}$ with the generalized Springer
correspondence.

\section{Preliminaries} \label{s.preliminaries}
In this paper we work over the ground field $\C$.
If $X$ is a complex algebraic variety, $R(X)$
will denote the ring of regular functions on $X$, and $\k(X)$ the
field of fractions of $R(X)$.  If $X$ smooth and of pure dimension $n$, $\Omega_X$ denotes the sheaf of
top degree differential forms (i.e.,~differential $n$-forms) on $X$.  This is
the sheaf of sections of the top exterior power of the cotangent bundle $T^*_X$.
Since we will have no use in this paper for forms of less than
top degree, we will generally omit the superscript and write $\Exterior T^*_X = \Exterior^n T^*_X$.

If $H$ is a linear algebraic group, $H_0$
will denote its identity component.  We denote the Lie algebra of an algebraic
group by the corresponding fraktur letter.  If $H$ acts on $X$, the stabilizer
in $H$ of $x \in X$ is denoted by $H^x$.  If $V$ is a vector space, let $V_X$ denote the
vector bundle $V \times X \to X$; if $V$ is a representation of $H$, then $V_X$ is an $H$-equivariant
vector bundle.
The orbit $H \cdot x$ is open in its
closure $\overline{H \cdot x}$ (see \cite[Section 2.1]{Jan:04}).  We will make use of the following lemma.

\begin{Lem} \label{lem.orbits}
Suppose $f:X \to Y$ is an $H$-equivariant map of irreducible varieties of the same dimension, and suppose that
$H \cdot y$ is open in $Y$.  Let $x \in f^{-1}(y)$.  Then $f^{-1}(H \cdot y) = H \cdot x$.  
\end{Lem}

\begin{proof}
Since $x \in f^{-1}(y)$, we have $H \cdot x \subseteq f^{-1}(H \cdot y)$.  We prove the reverse inclusion.
Since $f$ maps $H \cdot x$ surjectively onto $H \cdot y$, which has the same dimension as $Y$ and $X$,
it follows that $\dim H \cdot x = \dim X$.  Hence $\overline{H \cdot x} = X$, so $H \cdot x$ is open in $X$.
The complement of $H \cdot x$ in $X$ consists of orbits of smaller dimension than $\dim H \cdot y$, so no such orbit can
map $H$-equivariantly to $H \cdot y$.  Hence $H \cdot x \supseteq  f^{-1}(H \cdot y)$, as desired.  
\end{proof}

\subsection{Tori} \label{ss.prelimtori}
If $T$ is an algebraic torus with Lie algebra $\ft$, $\widehat{T}$ will denote the character group of $T$.  This is a free abelian
group which can be viewed as a subset of $\ft^*$.  We generally use a Greek letter to denote an element of
$\widehat{T}$  when viewed as an element of $\ft^*$, and exponential notation to denote the same element
viewed as a function on $T$, e.g. $\gl$ and $e^{\gl}$.  Thus, $R(T)$ is the span of $e^{\gl}$ for $\gl$ 
in the lattice $\widehat{T}$.  Under the action of $T$ on $R(T)$, $e^{\gl}$ is a weight vector of weight $-\gl$,
since for $s, t \in T$, we have
$$
(t \cdot e^{\gl})(s) = e^{\gl}(t^{-1} s) = e^{\gl}(t^{-1}) e^{\gl} (s) = e^{-\gl}(t) e^{\gl}(s).
$$
Let $\V^* = \widehat{T} \otimes_{\Z} \R$, and $\V$ the dual real vector space 
(the notation is motivated by the conventions of \cite[Section 1.2]{Ful:93}).  
We can identify $\V^*$ with the real span of $\widehat{T}$ in $\ft^*$.  

Throughout this paper $Z$ will denote a finite subgroup of $T$, and $T_{ad} = T/Z$ the quotient
torus.  For most of this paper, $T$ and $T_{ad}$ will be maximal tori of the algebraic groups $G$ and $G_{ad}$,
and $Z$ will be the center of $G$,
but in Section \ref{ss.toric1}, $T$ and $Z$ can be arbitrary.   
Write ${\mathcal P} =
\widehat{T}$ for the character group of $T$  and
${\mathcal Q} = \widehat{T}_{ad}$.  We have $\cp \supset \cq$, 
and the quotient $\cp/\cq$ is isomorphic to the character
group $\hat{Z}$ of $Z$.  Since $\cp$ and $\cq$ span the same real subspace of $\ft^*$, we can identify
$\V = \cp \otimes_{\Z} \R = \cq \otimes_{\Z} \R$.

\subsection{Sheaves on mixed spaces} \label{ss.induction}
The main construction of this paper is an example of a ``mixed space".  By this we mean
a scheme of the form $G \times^H S$, where $G \supset H$ are linear
algebraic groups, $H$ acts on the scheme $S$, and $G \times^H S = (G \times S)/H$ where
$H$ acts on $G \times S$ by the mixing action: $h(g,s) = (g h^{-1}, hs)$.
In this subsection we discuss some results about mixed spaces which will be useful in studying canonical sheaves.
We provide some proofs for lack of a convenient reference.
We will use the following notation.   If $S$ is a scheme with an $H$-action, there is an equivalence between the
category of $H$-equivariant coherent sheaves on $S$ and the category of $G$-equivariant coherent sheaves
on $G \times^H S$.  We refer to this as an induction equivalence, and denote this equivalence by $\CInd_H^G$, so
if $\cf$ is an $H$-equivariant coherent sheaf on $S$, we denote by $\CInd_H^G \cf$ the corresponding
$G$-equivariant sheaf on $G \times^H S$.  

The induction equivalence is compatible with tensor products and pullbacks of sheaves, as well as direct images and higher direct images.
This can be seen as follows.  
Let $p$ and $q$ be the projections from $G \times S$ to (respectively) $G \times^H S$ and
$S$.  These projections are faithfully flat.  
The pullback $p^*$   induces an  equivalence of categories between $G$-equivariant sheaves on $G \times^H S$ and $G \times H$-equivariant
sheaves on $G \times S$.  The pullback $q^*$   induces an equivalence of categories between $H$-equivariant sheaves on $S$  and $G \times H$-equivariant
sheaves on $G \times S$.  See \cite[Section 6]{Tho:87}.  The induction equivalence is characterized by the equation
$p^* \CInd_H^G \cf = q^* \cf$.  Compatibility of the induction equivalence with tensor products holds because
pullback of sheaves commutes with tensor product; compatibility with pullbacks follows from
the functoriality of pullbacks; the compatibility with (higher) direct images 
follows from the compatibility of higher direct images with flat pullback (\cite[III, Prop.~9.3]{Har:77}).

Let $e^{\gl}: H \to \G_m$ be a homomorphism, and let $\C_{\gl}$ denote the corresponding $1$-dimensional
representation of $H$.  Let $\cl_{\gl}$ denote the sheaf of sections of the line bundle $G \times^H \C_{\gl} \to G/H$
on $G/H$.  We abuse notation and view $\C_{\gl}$ as an $H$-equivariant sheaf over a point;
then $\cl_{\gl} = \CInd_H^G(\C_{\gl})$.

If $\cf$ is an $H$-equivariant sheaf on $S$, we abuse notation and write $\cf \otimes \C_{\gl}$ for the $H$-equivariant
sheaf $\cf \otimes p_S^* \C_{\gl}$, where $p_S$ is the projection from $X$ to a point.  The sheaf $\cf \otimes \C_{\gl}$ can be identified with the sheaf $\cf$, but with action twisted
by $\C_{\gl}$: precisely, if $s \in \cf(U)$ for some open set $U$ of $S$, then $s \otimes 1 \in (\cf \otimes \C_{\gl})(U)$,
and $h(s \otimes 1) = e^{\gl}(h)((hs) \otimes 1)$.  Given an $H$-equivariant map $S \to S'$, we have
$f_* (\cf \otimes \C_{\gl}) = f_* \cf \otimes \C_{\gl}$.
Let $\pi: G \times^H S \to G/H$ denote the projection.

\begin{Lem} \label{lem.twistinduction}
With notation as above, we have
$$
\CInd_H^G (\cf \otimes \C_{\gl}) = (\CInd_H^G \cf) \otimes \pi^* \cl_{\gl}
$$
as $G$-equivariant sheaves on $G \times^H S$.
\end{Lem}

\begin{proof}
By definition, $\CInd_H^G (\cf \otimes \C_{\gl}) = \CInd_H^G (\cf \otimes p_S^* \C_{\gl})$.  We have
$$
 \CInd_H^G (\cf \otimes p_S^* \C_{\gl})
= \CInd_H^G (\cf) \otimes \CInd_H^G (p_S^* \C_{\gl}) = \CInd_H^G (\cf) \otimes \pi^* \CInd_H^G( \C_{\gl})
= (\CInd_H^G \cf) \otimes \pi^* \cl_{\gl};
$$
here, the first equality uses the compatibility of induction with tensor product, and the second equality uses the compatibility
of induction with pullback.
\end{proof}

We will make use of the following lemma about the canonical sheaf of a mixed space.

\begin{Lem} \label{lem.mixedcanonical}
With notation as above, suppose $S$ is smooth, and let $\pi: G \times^H S \to G/H$ be the projection.  Then
as $G$-equivariant sheaves on $G \times^H S$, we have
$$
\Omega_{G \times^H S} = \CInd_H^G \Omega_S \otimes \pi^* \Omega_{G/H}.
$$
\end{Lem}

\begin{proof}
 With $p$ and $q$ as above, we have an $G \times H$-equivariant
 exact sequence of vector bundles on $G \times S$, where the first map in the sequence
 is the composition $q^* TS \to T(G \times S) \to p^*T(G \times^H S)$:
$$
0 \to q^*( TS) \to p^* T(G \times^H S) \to p^* \pi^* T(G/H) \to 0.
$$
By the equivalence of categories
discussed above, this exact sequence induces a $G$-equivariant exact sequence of 
vector
bundles on $G \times^H S$:
$$
0 \to V \to T(G \times^H S) \to \pi^* T(G/H) \to 0,
$$
where $p^* V = q^*(TS)$.
This implies that 
$$
\bigwedge T^*(G \times^H S) = \bigwedge V^* \otimes \pi^* \bigwedge T^*(G/H).
$$
Therefore the sheaves of sections of these line bundles are equal; this is the
statement of the lemma (note that the sheaf of sections of $\bigwedge V^*$ is $\CInd_H^G \Omega_S$).
\end{proof}

\subsection{Semisimple groups} \label{ss.prelimsemisimple}
For the rest of this paper, $G$ will denote a complex semisimple simply connected algebraic group with center $Z$.
The corresponding adjoint group is $G_{ad} = G/Z$.

We introduce some notation related to nilpotent elements which will be in effect for the remainder
of the paper except for Section \ref{s.otherorbits}, since in that section we consider other nilpotent orbits
besides the principal (that is, regular) nilpotent orbit.
Let $\nilp$ denote a principal nilpotent element in $\fg$; this means
that the principal nilpotent orbit $G \cdot \nilp = \co^{pr} \cong G/G^{\nilp}$ is a dense open
subvariety of $\cn$.  The universal cover of $\co^{pr}$ is $\tilde{\co}^{pr} = G/G^\nilp_0$.
Write $\tnilp \in \tilde{\co}^{pr}$ for the coset $1 \cdot
G^{\nilp}_0$; then $G^{\tnilp} = G^{\nilp}_0$.

Choose a standard ${\mathfrak{sl}}_2$-triple $\{h, e, f\}$ such that
$e = \nilp$ is the nilpositive element.   Write $\fg_i$ for the $i$-eigenspace of $\ad h$ on $\g$; then
$\fg = \oplus \fg_i$, and $\fg_i = 0$ if $i$ is odd (in other words, $\nilp$ is an even nilpotent element).   Also, $\nilp \in \fg_2$.  
Write $\fg_{\geq k} = \oplus_{i \geq k} \fg_i$ and $\fg_{>k} =  \oplus_{i > k} \fg_i$.
Let $\ft = \fg_0$, and let $T$ denote
the corresponding connected subgroup of $G$.  Then $T$ is a maximal torus of $G$, and $T_{ad} = T/Z$ is
a maximal torus of $G_{ad}$.  In this setting,  ${\mathcal P} =
\widehat{T}$ is the weight lattice of $G$, and ${\mathcal Q} = \widehat{T}_{ad}$ is the root lattice.
Let $\fu =\fg_{>0}$;
then $\fb = \ft + \fu$ is a Borel subalgebra of $\fg$.  We have $\fu_i = \fg_i$ if $i>0$.
Let $B$ and $U$ be the subgroups of $G$
whose Lie algebras are $\fb$ and $\fu$, respectively.  By \cite{BaVo:85}, we have $G^{\nilp} = B^{\nilp}
= T^{\nilp} U^{\nilp}$.  Because $\nilp$ is principal, $T^{\nilp}$ is equal to the center $Z$ of $G$.  We let
$W$ denote the Weyl group of $T$ in $G$.

Let $\Phi$ denote the set of roots of $\ft$ in $\fg$, and let $\Phi^+$ denote the positive root system 
whose elements are the weights of $\ft$ on $\fu$.  Let $\ga_1, \ldots, \ga_r$ denote the simple roots, and
let $\xi$ denote the sum of the simple roots.  
We can choose vectors
$E_{\alpha_i}$ in the root spaces $\g_{\alpha_i}$ 
such that $\nilp = \Sigma E_{\alpha_i}$ \cite{Kos:59}.  The vectors
$ E_{\alpha_1}, \ldots, E_{\ga_r}$ form a basis of $\fg_2$.  Let $ E_{\ga_1}^*, \ldots,  E_{\ga_r}^*$ denote the dual basis 
of $\fg_2^*$.
Then $E_{\ga_i}^* $ is a weight vector of weight $-\ga_i$.  The ring $R(\fg_2)$ is isomorphic
to the polynomial ring $\C[E_{\ga_1}^*, \ldots, E_{\ga_r}^*]$.  If we embed $T_{ad}$ into $\fg_2$ via
the map $t \mapsto t \cdot \nilp$, then we obtain an inclusion $R(\fg_2) \hookrightarrow R(T_{ad})$ satisfying
$E_{\ga_i}^* \mapsto e^{\ga_i}$.

\section{The toric varieties $\toric$ and $\toric_{ad}$} \label{s.toric}
The main construction of this paper uses toric varieties $\toric$ and $\toric_{ad}$.  In this section we introduce
these varieties and study some of their basic properties.  

\subsection{Toric varieties and finite quotients} \label{ss.toric1}
In this section $T$ is an arbitrary torus and $Z$ is a finite subgroup of $T$.
The relationship $T/Z = T_{ad}$ extends to a relationship between toric varieties for the tori 
$T$ and $T_{ad}$.  
 
Recall from Section \ref{ss.prelimtori} that we write $\cp = \hat{T} \supset \cq = \hat{T}_{ad}$. 
If $S$ is any subset of $P$, let $\C[S]$ be the subspace of $R(T)$ spanned by the $e^{\gl}$
for $\gl \in S$.  If $S$ is closed under addition, then $\C[S]$ is a subring of $R(T)$.

Suppose $\sigma$ is a rational polyhedral cone
in $\V$, with dual cone $\sigma^{\vee}$ in $\V^*$.  
Corresponding to $\sigma$, there are affine
toric varieties $X(\sigma)$ and $X_{ad}(\sigma)$, for the tori $T$ and $T_{ad}$, respectively.
We have semigroups $\sigma^{\vee}_{\cp} = \sigma^{\vee} \cap \cp$ and $\sigma^{\vee}_{\cq} = \sigma^{\vee} \cap \cq$
(in \cite{Ful:93}, the semigroup is denoted $S_{\sigma}$, but since we are considering two tori we use
a different notation so that we can distinguish the corresponding semigroups).
By definition, $R(X(\sigma)) = \C[\sigma^{\vee} \cap \cp ]$ and $R(X_{ad}(\sigma)) = \C[\sigma^{\vee} \cap \cq ]$.
(The notation $\chi^{\gl}$ is used in \cite{Ful:93} for what we denote as $e^{\gl}$.)

The $T$-equivariant inclusion $R(X_{ad}(\sigma)) \hookrightarrow R(X(\sigma))$ induces
a $T$-equivariant map $\tormap: X(\sigma) \to X_{ad}(\sigma)$.  The first part of the next
proposition is essentially in  \cite[Section 2.2]{Ful:93}.

\begin{Prop} \label{p.toric}  
\begin{enumerate}
\item The natural inclusion $R(X_{ad}(\sigma)) \hookrightarrow R(X(\sigma))$ has image equal
to $R(X(\sigma))^Z$, so the corresponding map $\tormap: X(\sigma) \to X_{ad}(\sigma)$ is the quotient by $Z$.

\item Under this map, the inverse image of 
the inverse image of $T_{ad}$ is $T$.  Hence the
inverse image of $1 \in T_{ad}$ is $Z \subset T$.
\end{enumerate}
\end{Prop}

\begin{proof}
If $\gl \in \cp$ then any element $Z$ takes $e^{\gl}$ to a multiple of itself, and $e^{\gl}$ is
fixed by $Z$ if and only if $\gl \in \cq$.  This implies that  $R(X(\sigma))^Z$ is the span of
$e^{\gl}$ for $\gl \in \sigma^{\vee} _{\cq}$, which is exactly $R(X_{ad}(\sigma)) $.  This proves (1).  Since the map is finite and
$T$-equivariant, the unique orbit in $X(\sigma)$ mapping to the open orbit $T_{ad}$ in $X_{ad}(\sigma)$
is the open orbit $T$.  This proves (2).
\end{proof}

The previous proposition shows that the fibers of $\tormap: X(\sigma) \to X_{ad}(\sigma)$ 
over the open $T_{ad}$-orbit are isomorphic to $Z$.  We now describe the fibers of $\tormap$ over the other
$T_{ad}$-orbits in $X_{ad}(\sigma)$.  We need some facts about toric varieties, and in particular
 the description of orbits, from
\cite[Section 3.1]{Ful:93}. 
The $T$-orbits in $X(\sigma)$ correspond to faces of $\sigma$.  If $\tau$ is a
face of $\sigma$, the orbit $\co(\tau)$ is isomorphic to the torus $T(\tau) = \Spec \C[\tau^{\perp} \cap \cp]$.
If $\tau = \{0\}$ then $T(\tau) = T$ and
$\co(\tau)$ is the open $T$-orbit. 
The closure of $\co(\tau)$, denoted by $V(\tau)$, is an affine toric variety for the torus $T(\tau)$, and
$R(V(\tau)) = \C[ \tau^{\perp} \cap \sigma^{\vee}_{\cp}]$.  The embedding of $V(\tau)$ into
$X(\sigma)$ corresponds to the map on rings of regular functions
$$
\C[\sigma^{\vee}_{\cp}] \to \C[ \tau^{\perp} \cap \sigma^{\vee}_{\cp}]
$$
taking $e^{\gl}$ to $e^{\gl}$ if $\gl \in \tau^{\perp}$, and $0$ otherwise.

The analogous picture holds for the
$T_{ad}$-toric variety $X_{ad}(\sigma)$; we use the
subscript $ad$ for the analogous definitions in this case.  We have a commuting diagram
of coordinate rings, where all maps are inclusions:
\begin{equation} \label{e.CDtorus1}
\begin{CD}
\C[ \tau^{\perp} \cap \cp] @<<< \C[ \tau^{\perp} \cap \sigma^{\vee}_{\cp}] \\
@AAA    @AAA \\
\C[ \tau^{\perp} \cap \cq] @<<< \C[\tau^{\perp} \cap \sigma^{\vee}_{\cq}] \\
\end{CD}
\end{equation}
Applying $\Spec$ to this diagram yields the following diagram of tori and orbit closures:
\begin{equation} \label{e.CDtorus2}
\begin{CD}
T(\tau) @>>> V(\tau) \\
@VVV    @VVV \\
T_{ad}(\tau) @>>> V_{ad}(\tau).
\end{CD}
\end{equation}

Define $Z(\tau) = {\rm ker}(T(\tau) \rightarrow T_{ad}(\tau))$.  This
is a finite group; if $\tau= \{0\}$ then $Z(\tau) = Z$.  The character
group $\widehat{Z}(\tau)$ is given by $\widehat{Z} (\tau) =
\widehat{T}(\tau)/\widehat{T}_{ad}(\tau) = (\tau^{\perp} \cap \cp) / (\tau^{\perp} \cap \cq)$.
Note that the groups $Z(\tau)$ and $\widehat{Z}(\tau)$ are isomorphic.

\begin{Prop}  \label{p.orbits}
(a) The map ${\mathcal O}(\tau) \rightarrow {\mathcal
O}_{ad}(\tau)$ is a covering map with fibers $Z(\tau)$.

(b) The map $V(\tau) \rightarrow V_{ad}(\tau)$ is an isomorphism $\Leftrightarrow$ the map is an
isomorphism on the open orbits, i.e., ${\mathcal O}(\tau) \rightarrow
{\mathcal O}_{ad}(\tau)$ is an isomorphism.
\end{Prop}

\begin{proof} (a) holds because the orbits are isomorphic to tori, and the map of orbits corresponds to the map of
tori.  For (b), by \eqref{e.CDtorus1}, we see that the map $V(\tau) \rightarrow V_{ad}(\tau)$ is an isomorphism
exactly when  $\sigma^{\vee} \cap \tau^{\perp} \cap \cp = \sigma^{\vee} \cap \tau^{\perp} \cap \cq$, and the map ${\mathcal O}(\tau) \rightarrow {\mathcal O}_{ad}(\tau)$ is an isomorphism exactly when exactly when $ \tau^{\perp} \cap \cp =  \tau^{\perp} \cap \cq$.  
The implication $(\Leftarrow)$ follows
immediately from this.  The implication $(\Rightarrow)$ follows from the fact that if $\gl$ is in the relative interior
of $\sigma^{\vee} \cap \tau^{\perp} \cap \cq$, the rings $\C[ \tau^{\perp} \cap \cp]$ and
$\C[ \tau^{\perp} \cap \cq]$ are obtained (respectively) from $\C[\tau^{\perp} \cap \sigma^{\vee}_{\cp}]$ and
 $\C[\tau^{\perp} \cap  \sigma^{\vee}_{\cq}]$ by adjoining the element $e^{-\gl}$ (cf.~\cite{Ful:93}, Section 1.1, and Sec.~1.2,
 Prop.~2).
\end{proof}

\subsection{The toric varieties $\toric_{ad} = \fu_{2}$ and $\toric$} \label{s.toric.g2}
We adopt the notation of Section \ref{ss.prelimsemisimple}.
Let $\toric_{ad}$ denote the vector space $\fu_2 = \fg_2$.  The map $T_{ad} \to T_{ad} \cdot \nilp$
embeds $T_{ad}$ as a dense orbit in $\toric_{ad}$, and $\toric_{ad}$ is an affine toric variety for $T_{ad}$.
The nilpotent element $\nilp$ corresponds to the identity element $1 \in T_{ad}$.
As noted in Section \ref{ss.prelimsemisimple}, $R(\toric_{ad}) = \C[E_{\ga_1}^*, \ldots, E_{\ga_r}^*]$.  
When viewed as a function on $T$, $E_{\ga_i}^*$
is the function $e^{\ga_i}$, which is a weight vector of weight $-\ga_i$.  

Let $\sigma^{\vee} \subset \V^*$ be the cone corresponding to $\toric_{ad}$.
The set of simple roots $\{ \ga_i \}$ is a basis of $\V^*$, and
$\sigma^{\vee}$ is the cone generated by  
the elements of the basis $\{ \ga_i \}$ of simple roots.  The dual cone $\sigma$ is generated by the 
elements of the dual basis $\{ v_i \}$ of $\V$.

Let $\toric$ denote the affine toric variety for $T$ defined using the cone $\sigma$, but using the lattice $\cp$
instead of $\cq$.   Thus,  $R(\toric) = \C[\sigma^{\vee}_{\cp}]$ and
$R(\toric_{ad}) = \C[\sigma^{\vee}_{\cq}]$.  

Since in general the cone $\sigma^{\vee}$ is not generated by part of a basis for $\cp$ (see Example \ref{ex.sl4}),
the cone $\sigma$ is not generated by part of a basis for $\cp^{\vee}$.  
It follows by \cite[Section 2.1]{Ful:93} that the variety $\toric$ is singular in general; we will see below that it is a quotient of a smooth variety by a finite group.
Any toric variety has a resolution of singularities which is itself a toric variety
(for the same torus); such
a resolution can be obtained by taking the toric variety associated to an appropriate subdivision of the fan of the original
toric variety (see \cite[Section 2.6]{Ful:93}).  In subsequent sections, 
we will let $\widehat{\toric} \to \toric$
denote a resolution of singularities of $\toric$ obtained in this way; note that such a resolution is proper.

In each coset of $\P \mbox{ mod } \Q$, we consider two
distinguished elements:
\begin{enumerate}
\item  The minimal dominant weight $\lambda_{dom}$
\item The element $\lambda_R = \Sigma a_{\alpha} \alpha$ with
$0 \le a_{\alpha} < 1$.
\end{enumerate}
 Note that for the identity coset, our convention is that $\lambda_{dom} = \gl_R = 0$.

It sometimes happens that the elements $\lambda_{dom}$ and $\lambda_R$ corresponding to the same
coset are equal.  Remarkably,  they are always conjugate by the Weyl group (see Proposition \ref{p.conjugacy}).
The letter $R$ in the notation $\lambda_R$ is chosen because these
weights are related to $R(\toric) = \C[\sigma^{\vee}_{\cp}]$ by the following
proposition (whose proof is immediate).

\begin{Prop}  \label{p.lambdar}
The semigroup $\sigma^{\vee}_{\P}$ is generated by
the simple roots $\alpha$ and the $\lambda_R$.  Hence
$R(\toric)$ is a free $R(\toric_{ad})$-module with basis given by the
$e^{\lambda_R}$.  $\Box$
\end{Prop}

\medskip

\begin{example}  \label{ex.sl4}
For $\fg = \mathfrak{sl}_4$, the minimal dominant weights are $\lambda_1 = \frac{1}{4}
(3\alpha_1 + 2\alpha_2 + \alpha_3)$, $\lambda_2 = \frac{1}{2}(\alpha_1 + 2\alpha_2 + \alpha_3)$, $\lambda_3 = \frac{1}{4}(\alpha_1 + 2\alpha_2 + 3\alpha_3)$.  So 
$\lambda_{1,R} = \lambda_1$, $\lambda_{2,R} = \frac{1}{2}(\alpha_1 + \alpha_3)$, $\lambda_{3,R} = \lambda_3$.  Note that
under the usual identification of $\ft^*$ with the subspace of $\C^4$ where all coordinates sum to zero, and
$\ga_i = \gre_i - \gre_{i+1}$, we have $\gl_2 = (\frac{1}{2}, \frac{1}{2}, -\frac{1}{2}, -\frac{1}{2} )$ and
$\gl_{2,R} =  (\frac{1}{2}, -\frac{1}{2}, \frac{1}{2}, -\frac{1}{2})$.  We see that $\gl_2$ and $\gl_{2,R}$ are indeed conjugate
by the Weyl group $S_4$, as asserted above.  Write
$x_i = e^{\ga_i}$ and $w_i = e^{\gl_{i,R}}$.  
Then
\begin{eqnarray*}
R(T_{ad}) & = & \C[x_1^{\pm 1},  x_2^{ \pm1},  x_3^{ \pm 1} ]\\
R(T) & = & \C[ w_1^{\pm 1}, w_2^{\pm 1}, w_3^{\pm 1} ].
\end{eqnarray*}
The embedding $R(T_{ad}) \hookrightarrow R(T)$ is given by
$x_1 \mapsto \frac{w_1 w_2}{w_3}$, $x_2 \mapsto \frac{w_1 w_3}{w_2^2}$,
$x_3 \mapsto  \frac{w_2 w_3}{w_1}$.  This follows from the equations
$\alpha_1 =
\lambda_{1,R} + \lambda_{2,R} - \lambda_{3,R}$; $\alpha_2 =
\lambda_{1,R} - 2\lambda_{2,R} + \lambda_{3,R}$; $\alpha_3 =
-\lambda_{1,R} + \lambda_{2,R} + \lambda_{3,R}$.  We have
\begin{eqnarray*}
R(\toric_{ad}) &=& \C[x_1, x_2, x_3]\\
R(\toric) &=& \C[w_1, w_2, w_3,x_1, x_2, x_3] = \C[w_1, w_2, w_3,
\frac{w_1w_2}{w_3}, \frac{w_1w_3}{v^2_2}, \frac{w_2w_3}{w_1}] .
\end{eqnarray*}
Note that when we write $R(\toric) = \C[w_1, w_2, w_3,x_1, x_2, x_3]$, we are not asserting
that $R(\toric)$ is a polynomial ring in the variables $w_i$ and $x_i$; we are describing
$R(\toric)$ as a subring of $R(T)$, where the $x_i$ are expressed in terms of the $w_i$
as described above.
\end{example}

\medskip

\begin{Prop} \label{p.orbifold}
There is a torus $T_d$ with a surjective map $T_d \to T$ and finite kernel $Z_d$, and a $T_d$-toric variety
$\toric_d \cong \C^r$ such that $\toric_d/Z_d \cong \toric$.  Hence $\toric$ is an orbifold.
\end{Prop}

\begin{proof}
Let $d$ be a positive integer such that $\cp$ is contained in the lattice
$\cq_d = \frac{1}{d} \cq$.  (For example, if $\fg = \mathfrak{sl}_n$, then $d=n$ suffices.)
Let $T_d$ be the torus whose character group $\widehat{T}_d$ is $\cq_d$;
let $Z_d$ be the kernel of the surjective map $T_d \to T$.  The character group
$\widehat{Z}_d$ is isomorphic to $\cq_d/\cp$.  Let $\toric_d$ be the affine toric variety
for $T_d$ corresponding to the cone $\sigma^{\vee}$.  The semigroup $\sigma^{\vee}_{\cq_d}$
is generated by the elements $\frac{1}{d} \ga_1, \ldots, \frac{1}{d} \ga_r$.  Since these elements
form a basis for $\cq_d$, the toric variety $\toric_d$ is isomorphic to $\C^r$.  Moreover, by the
arguments in Section \ref{ss.toric1},
we have $\toric_d/Z_d = \toric$. 
\end{proof}

The fibers of the map $\pi: \toric \rightarrow \toric_{ad}$ can be calculated using Proposition
\ref{p.orbits}.  As noted above, the cone $\sigma$ in $\V$ 
is generated by the $v_i$.  The
faces of $\sigma$ correspond to subsets $J$ of $\{1, \dots, n\}$, the
face $\tau_J$ being generated by the $v_j$ with $j \in J$.

One can describe $\widehat{Z}(\tau_J) = ( \tau^{\perp}_J \cap
\widehat{T} ) / ( \tau^{\perp}_J \cap \widehat{T}_{ad})$ in terms of
$J$.  If $J$ is empty then $\widehat{Z}(\tau) = \widehat{Z}$; in
general there is a natural injection of $(\tau^{\perp}_J \cap
\widehat{T}) / ( \tau^{\perp}_J \cap \widehat{T}_{ad} )$ into
$\widehat{T}/\widehat{T}_{ad} = \widehat{Z}$.  The image consists of
the cosets $\lambda + \widehat{T}_{ad}$ which have nonempty
intersection with $\tau_J^{\perp}$, i.e., those $\lambda +
\widehat{T}_{ad}$ such that if $\lambda = \Sigma a_i \alpha_i$, then
for all $j \in J$, we have $a_j \in \Z$.  The cosets of
$\widehat{T}_{ad}$ in $\widehat{T}$ have as representatives the
minimal dominant weights; these and the corresponding $a_i$ are listed
in \cite[Section 3.13]{Hum:72}.  From this list we can determine
$\widehat{Z}(\tau)$, or equivalently $Z(\tau)$.  We state the answer
as the next proposition, using the numbering of simple roots from
\cite{Hum:72}.  The statement of this proposition in an earlier version of this paper
contained some errors; the correct statement was given by Russell in \cite{Rus:12},
to which we refer for the proof.

\begin{Prop}  \label{prop.fibercalc}
Let $\toric$ and $\toric_{ad}$ be toric varieties associated to a simple Lie
algebra as above, and let $J$ be a nonempty subset of $\{1, \dots,
n\}$.  The group $Z(J) = Z(\tau_J)$ is given as follows.
\begin{equation*}
\begin{array}{lllll}
A_n: & Z(J) &=& \Z/c \Z & {\rm where \ } c = \gcd(J \cup \{ n + 1 \} )\\ 
B_n: & Z(J) &=& \Z/2\Z  & {\rm if \ all  \ }  j \in J {\rm \ are \ even} \\ 
         & Z(J) &=& \{1\}    &  {\rm otherwise } \\ 
C_n: & Z(J) &=& \Z/2\Z & {\rm if \ } n \not\in J \\
         & Z(J) &=& \{1\}    &  {\rm otherwise} \\ 
D_n: & Z(J) &=& Z         & {\rm if  \ } n - 1, n \not\in J {\rm \ and \ all \ } j \in J {\rm \ are \ even} \\
         & Z(J) &=& \Z/2\Z & {\rm if  \ } n - 1, n \not\in J {\rm \ and \ not \ all \ } j \in J {\rm \ are \ even} \\
         & Z(J) &=& \Z/2\Z & {\rm if \ exactly \ one \ of \ } n - 1 {\rm \ and \  } n {\rm \ is \ in \ }J ,\\
         &         &  &            & {\rm all \ }j \in J {\rm \ such \ that \ } j < n-1 {\rm \ are \ even },\\
          &         &  &            & {\rm and \ } n=4k+2 {\rm \ for \ some \ } k \geq 1 \\
         & Z(J) &=& \{1\}     & {\rm otherwise}\\ 
E_6: & Z(J) &=& \Z/3\Z   & {\rm if \ none \ of \  } 1, 3, 5, 6 {\rm \ are \ in \ } J; {\rm \ otherwise \ } Z(J) = \{1\} \\ 
E_7: & Z(J) &=& \Z/2\Z   & {\rm if \ none \ of \ }  2, 5, 7 {\rm \ are \ in \ } J; {\rm \ otherwise \ } Z(J) = \{1\}.
\end{array}
\end{equation*}
\end{Prop}

%--------------------------------------------------------

\section{The generalization of the Springer resolution} \label{s.utilde}

In this section we construct the variety $\tilde{\cm}$, and study its properties.  In particular, we show that there is
a commutative diagram
$$
\begin{CD}
\tilde{\cm} @>{\mnmap}>> \tilde{\cn} \\
@V{\mmap}VV         @VV{\mu}V \\
\cm @>{\eta}>> \cn,
\end{CD}
$$
where the map $\mmap$ is proper and an isomorphism over $\widetilde{\co}^{pr}$ (see Theorem \ref{t.genspringer}).

\subsection{The variety $\tilde{\u}$} \label{s.tildeu}
By definition, $\toric_{ad} = \fu_2 \cong \fu/\fu_{\geq 4}$.  (Recall that $\fu_i = \fg_i$ for $i>0$.)  With the interpretation
of $\toric_{ad}$ as $\fu/\fu_{\geq 4}$, since the $B$-action on $\fu$ preserves $\fu_{\geq 4}$, the variety $\toric_{ad}$ has a $B$-action.  
Under this action,
the unipotent radical $U$ acts trivially, and the projection $\fu \to \toric_{ad}$ is $B$-equivariant.  Consider the maps
\begin{equation} \label{e.toricmaps}
\widehat{\toric} \stackrel{\widehat{\rho}}{\longrightarrow} \toric \stackrel{\rho}{\longrightarrow} \toric_{ad},
\end{equation}
where, as in Section \ref{s.toric.g2}, $\widehat{\toric}$ is a toric resolution
of singularities of $\toric$.  
These maps are $T$-equivariant, and become $B$-equivariant if we extend the
$T$-actions on $\widehat{\toric}$ and $\toric$ to $B$-actions by requiring that $U$
act trivially.  We define $\widetilde{\fu} = \widetilde{\toric} \times_{\toric_{ad}} \fu$
and $\widehat{\fu} = \widehat{\toric} \times_{\toric_{ad}} \fu$.   Consider the following
diagram, where the squares are fiber squares:
\begin{equation} \label{e.fiber}
\begin{CD}
\widehat{\fu} @>{\widehat{\sigma}}>> \widetilde{\fu} @>{\sigma}>> \fu \\
@V{\widehat{p}_1}VV  @V{\widetilde{p}_1}VV    @V{p_1}VV \\
\widehat{\toric} @>{\widehat{\rho}}>> \toric @>{\rho}>> \toric_{ad} .
\end{CD}
\end{equation}
In this diagram, we have identified $ \toric_{ad} \times_{\toric_{ad}} \fu$ with $\fu$.
The maps $\widehat{p}_1$, $\widetilde{p}_1$, and $p_1$ are projections onto the first factors of the fiber
products.
Under our identification of $ \toric_{ad} \times_{\toric_{ad}} \fu$ with $\fu$,
the map
$\sigma: \tilde{\fu} = \toric \times_{\toric_{ad}} \fu \to \fu$ is simply the projection on the second factor.
We write $\gamma = \sigma \circ \widehat{\sigma}: \widetilde{\fu} \to \fu$.
The fiber products $\widehat{\fu}$ and $\widetilde{\fu}$ have $B$-actions coming from the $B$-actions on each factor,
and the maps above are all $B$-equivariant.  Also, $\widehat{\sigma}$ is proper, and $\sigma$
is finite, since these properties hold for the maps of toric varieties in \eqref{e.toricmaps}.
As algebraic varieties, we have
\begin{equation} \label{e.direct}
\tilde{\fu} \cong \toric \times \fu_{\ge 4}, \ \ \ \  \widehat{\fu} 
\cong \widehat{\toric} \times \fu_{\ge 4}.  
\end{equation}
The reason for defining $\tilde{\fu}$ and $\widehat{\fu}$
as fiber products as in \eqref{e.fiber}, rather than by the formulas of \eqref{e.direct}, is that the fiber product definition
allows us to equip $\tilde{\fu}$ and $\widehat{\fu}$ with a $B$-action.  

Extend the $Z$-action on $\toric$ to a $Z$-action on $\tilde{\u}$ by
$z \cdot (w, x) = (zw, x)$.  

\begin{Prop}  \label{p.bzcommute} 
\begin{enumerate}
\item The actions of $B$ and $Z$ 
on $\tilde{\fu}$ commute, and $\tilde{\u}/Z \cong \fu$ as $B$-varieties.

\item $R(\tilde{\fu})$ is the integral closure of $R(\fu)$ in $\kappa(\tilde{\fu})$.
\end{enumerate}
\end{Prop}

\begin{proof} (1) The actions of $B$ and $Z$ on $\toric$ and on $\fu$ commute (since the
$B$-action on $\toric$ comes from the $T$-action and $Z \subset T$, and the $Z$-action on 
$\fu$ is trivial).  Hence the actions of $B$ and $Z$ on the fiber product $\tilde{\fu}
= \toric \times_{\toric_{ad}} \fu$ commute.  To see that $\tilde{\u}/Z \cong \fu$ as $B$-varieties,
observe that the varieties $\toric$, $\toric_{ad}$, $\widetilde{\fu}$ and $\fu$ are all affine,
and $R(\fu) = R(\toric) \otimes_{R(\toric_{ad})} R(\fu)$.  
The $B$-equivariant inclusion $R(\toric)^Z \to R(\toric)$ induces a $B$-equivariant map
\begin{equation} \label{e.bzcommute1}
R(\toric)^Z \otimes _{R(\toric_{ad})} R(\fu) \to  (R(\toric) \otimes_{R(\toric_{ad})} R(\fu) )^Z = R(\tilde{\fu})^Z .
\end{equation}
This map is an isomorphism: indeed, $R(\fu)$ is free over $R(\toric_{ad})$, and given a basis
$\{ b_i \}$  for $R(\fu)$ over $R(\toric_{ad})$, and elements $a_i$  of $R(\toric)$, the sum $\sum a_i \otimes b_i$ is 
$Z$-invariant if and only if each $a_i$ is $Z$-invariant.  Hence the corresponding map of
affine varieties
$$
(\toric \times_{\toric_{ad}} \fu)/Z \to (\toric/Z) \times _{\toric_{ad}} \fu
$$
is an isomorphism.  Since $\toric/Z = \toric_{ad}$, this gives an isomorphism
$$
\tilde{\fu}/Z \to \toric_{ad} \times _{\toric_{ad}} \fu = \fu,
$$
proving (1).

(2) Since $\tilde{\fu}/Z = \fu$, we have
$R(\tilde{\fu})^Z = R(\fu)$.  Hence  $R(\tilde{\fu})$ is integral over $R(\fu)$.  Since $\toric$ is a toric
variety, it is normal, hence so is $\tilde{\fu} \cong \toric \times \fu_{\geq 4}$.
Hence $R(\tilde{\fu})$ is integrally closed in $\kappa(\tilde{\fu})$.  It follows that $R(\tilde{\fu})$
is the integral closure of $R(\fu)$ in $\kappa(\tilde{\fu})$.
\end{proof}

\subsection{The variety $\cm = \Spec R(\tilde{\co}^{pr})$} \label{ss.cm}
In this subsection, we recall some known facts we will need about $\cm$ and $\cn$.  We include
some proofs for lack of a reference.

The principal nilpotent orbit $\co^{pr}$ is
open in $\cn$, and its complement has codimension $2$.  Since $\cn$ is normal 
(by \cite{Kos:63}), the restriction map $R(\cn) \to R(\co^{pr})$ is an isomorphism.
We identify $\co^{pr}$ with the
homogeneous variety $G/G^{\nilp}$ by the map $G/G^{\nilp} \to \co^{pr}$,
$g G^{\nilp} \mapsto G \cdot \nilp$.  

We define $\tilde{\co}^{pr}$ to be the homogeneous variety $G/G^{\nilp}_0$, and
write $\tilde{\nilp}$ for the coset $1 \cdot G^{\nilp}_0$ in $\tilde{\co}^{pr}$. 
Let $\cm = \Spec R( \tilde{\co}^{pr})$.  The inclusion $R(\co^{pr}) \hookrightarrow R( \tilde{\co}^{pr})$
induces a map $\eta: \cm \to \cn$.
Since the map 
$$
\tilde{\co}^{pr} = G/G^{\tilde{\nilp}} \to \co^{pr} \cong G/G^{\nilp}
$$
is $G$-equivariant, the map $\eta:\cm \to \cn$ is as well.
We claim that $\cm/Z = \cn$.  Indeed, this equation is equivalent to $R(\cm)^Z = R(\cn)$,
or equivalently, $R(\tilde{\co}^{pr})^Z = R(\co^{pr})$.  Since 
$R(\tilde{\co}^{pr}) = R(G)^{G^{\nilp}_0}$, and $G^{\nilp} = Z G^{\nilp}_0$, we have
$R(\tilde{\co}^{pr})^Z = R(G)^{G^{\nilp}}$, proving the claim.

The variety $\tilde{\co}^{pr}$ is quasi-affine (this is proved in the proof of Theorem 4.1 in \cite{McG:89}).  This can be seen
as follows (the argument of \cite{McG:89} can be simplified since we are only considering the
principal orbit).
Let $V$ be a representation of $G$ containing a highest weight vector $v$
such that the stabilizer group $Z^v$ is trivial.  Then the map $G/G^{\nilp}_0 \to \fg \times V$,
$g G^{\nilp}_0 \mapsto g \cdot (\nilp, v)$ embeds $\tilde{\co}^{pr}$ into $\fg \times V$, so
$\tilde{\co}^{pr}$ is quasi-affine.  Let $Y$ be the closure of $\tilde{\co}^{pr}$ in
$\fg \times V$.  Since any orbit is open in its closure,  $\tilde{\co}^{pr}$ is open in $Y$. The projection $\fg \times V \to \fg$ maps $Y$ to $\cn$, and takes $\tilde{\co}^{pr}$
to $\co^{pr}$.

\begin{Lem} \label{lem.normalization}
There is a map $\cm \to Y$ such that $\cm$ is the normalization of $Y$.  Hence there is
an open embedding $\tilde{\co}^{pr} = G/G^{\nilp}_0 \hookrightarrow \cm$.   Moreover, the restriction
map $R(\cm) \to R(\tilde{\co}^{pr})$ is an isomorphism.
\end{Lem}

\begin{proof}
Let $\kappa(\tilde{\co}^{pr})$ be the function field of $\tilde{\co}^{pr}$.  Since $\tilde{\co}^{pr}$ is quasi-affine,
$\kappa(\tilde{\co}^{pr})$ is the fraction field of $R(\tilde{\co}^{pr}) = R(\cm)$; this is also the
fraction field of $R(Y)$.   Because $\tilde{\co}^{pr}$ is normal
(being smooth), $R(\cm)$ is integrally closed in $\kappa(\tilde{\co}^{pr})$, so $\cm$ is normal.  
Because $R(\cm)^Z = R(\cn)$,
$R(\cm)$ is integral over $R(\cn)$ (\cite[Ch.~5, Exer.~12]{AtMa:69}).  The maps $\tilde{\co}^{pr} \to Y \to \cn$
yield maps
$$
R(\cn) = R(\co^{pr}) \to R(Y) \to R(\cm) = R(\tilde{\co}^{pr}).
$$
Since the composition is injective, the map $R(\cn) \to R(Y)$ is injective.  Since $R(\cm)$ is integral
over $R(\cn)$, it follows that $R(\cm)$ is integral
over $R(Y)$.  Since $R(\cm)$ is integrally closed, $R(\cm)$ is the integral closure of $R(Y)$ in its
fraction field.  Therefore the map $\cm \to Y$ induced by the inclusion $R(Y) \to R(\cm)$ is the normalization
map.  Since the normalization map is an isomorphism over the locus of nonsingular points, it is an isomorphism
over $\tilde{\co}^{pr}$.  Hence we obtain an open embedding of $\tilde{\co}^{pr}$ into $\cm$.
The map $\cm \to \cn$ is finite and takes $\widetilde{\co}^{pr}$ to $\co^{pr}$.  Since the complement of $\co^{pr}$
in $\cn$ has codimension $2$, so does the complement of $\widetilde{\co}^{pr}$ in $\cm$.  Since $\cm$ is normal
and the complement of $\widetilde{\co}^{pr}$ has codimension $2$, the restriction map $R(\cm) \to R(\widetilde{\co}^{pr})$ is
an isomorphism.
\end{proof}

By definition, $\tilde{\nilp}$  is the coset $1 \cdot G^{\tilde{\nilp}}$ in $\tilde{\co^{pr}}$, so, via the embedding of the
lemma, we view $\tilde{\nilp}$ as an element of $\cm$.  Under the composition $\cm \to Y \to \cn$,
we have $\tilde{\nilp} \mapsto (\nilp,v) \mapsto \nilp$.  The following corollary (known to
Brylinski and Kostant, who call $\cm$ the normal closure of $\widetilde{\co}^{pr}$) summarizes the
situation.

\begin{Cor} \label{cor.normalization}
There is a $G$-equivariant
commutative
diagram
 $$
 \begin{CD}
G/G^{\tilde{\nilp}}  @>=>> \tilde{\co}^{pr} @>>> \cm \\
@VVV  @VVV         @VV{\eta}V \\
G/G^{\nilp} @>{\cong}>>  \co^{pr} @>>> \cn,
 \end{CD}
 $$
 where the horizontal maps into $\cm$ and $\cn$ are open embeddings, and the vertical maps are
quotients by the center $Z$ of $G$.  $\Box$
\end{Cor}

%-------------

\subsection{The generalized Springer resolution}
We define $\tilde{\cm} = G \times^B \tilde{\fu}$
and $\widehat{\cm} = G \times^B \widehat{\fu}$.  Since $\tilde{\cn} = G \times^B \fu$, from our maps
$$
\widehat{\fu} \stackrel{\widehat{\sigma}}{\longrightarrow}  \widetilde{\fu} \stackrel{\sigma}{\longrightarrow} \fu
$$
we obtain maps
$$
\widehat{\cm} \stackrel{\widehat{\eta}}{\longrightarrow} \tilde{\cm}\stackrel{\widetilde{\eta}}{\longrightarrow} \tilde{\cn}.
$$
The map $\widehat{\eta}$ is proper (since $\widehat{\sigma}$ is proper) 
and $\widetilde{\eta}$ is finite (since $\sigma$ is finite).  In this section, we prove some properties of $\widetilde{\cm}$, construct a map $\tilde{\cm} \to \cm$, and study the properties of this map.

Let $Z_d$ be the finite group defined in the proof of Proposition \ref{p.orbifold}.
\begin{Prop} \label{prop.orbifold2}
There is a covering of $\widetilde{\cm}$ by open subvarieties $W_i$ of the form $W_i = \widetilde{W}_i/Z_d$, where
$\widetilde{W_i}$ is a smooth variety with an action of $Z_d$.
\end{Prop}

\begin{proof}
Because the group $B$ is solvable, it is a special group, so the principal bundle $\pi: G \to G/B$ is locally trivial in the Zariski topology (see \cite{Che:58}).  
This means we can cover $G/B$ with open subvarieties $W'_i$ such that $W_i = \pi^{-1}(W'_i) \cong W'_i \times B$.
Thus, $\widetilde{\cm}$ is covered by the subvarieties 
$$
(W'_i \times B) \times^B \widetilde{\fu} \cong W'_i \times  \widetilde{\fu} \cong W'_i \times \toric \times \fu_{\geq 4}.
$$
By Proposition \ref{p.orbifold}, $\toric = \toric_d/Z_d$, where $\toric_d \cong \C^r$ is smooth.  The result follows.
\end{proof}

By the preceding proof, $\widetilde{\cm}$ is locally the product of $\toric$ and a smooth variety.  Since
$\toric$ is not smooth in general, $\widetilde{\cm}$ is not smooth in general.

\begin{rem} \label{rem.fiber}
The fibers of $\widetilde{\eta}$
can be described explicitly.  For each subset $J$ of the simple roots,
one can define subsets $\widetilde{\cm}_J$ and $\widetilde{\cn}_J$ of $\widetilde{\cm}$ and $\widetilde{\cn}$ such
that $\eta^{-1}(\widetilde{\cn}_J) = \widetilde{\cm}_J$, and 
the fibers of $\widetilde{\eta}$ over $\widetilde{\cn}_J$ are isomorphic to the group $Z(J)$ calculated
in Proposition \ref{prop.fibercalc}.   
\end{rem}

In viewing $\toric_{ad}$ as a $T_{ad}$-toric variety, we used
the embedding $T_{ad} \hookrightarrow \toric_{ad}$ which takes $1$ to $\nilp$.  We also denote by $1$ the element in $\toric$ corresponding to the identity in
$T$ under $T \hookrightarrow \toric$. Then, under $\rho: \toric \rightarrow \toric_{ad}$,
we have $\rho(1) = 1$, and under $\sigma: \widetilde{\fu} \to \fu$, we have $\sigma(1, \nilp) = \nilp$.

Recall that $\tilde{\nilp} = 1 \cdot G^{\nilp}_0 \in \tilde{\co}^{pr} = G/G^{\nilp}_0$, and $G^{\tilde{\nilp}} = B^{\tilde{\nilp}}
=B^{\nilp}_0 = U^{\nilp}$.

\begin{Prop}  \label{prop.stabilizer}
\begin{enumerate}
\item The stabilizer groups $B^{(1, \nilp)}$ and $B^{\tilde{\nilp}}$ are
equal.

\item Under the map $\sigma: \tilde{\u} \rightarrow \u$, we have
$\sigma^{-1}(B \cdot \nilp) = B \cdot (1, \nilp)$.  Hence
$B \cdot (1, \nilp)$ is open  in $\tilde{\fu}$, and $\overline{B \cdot (1, \nilp)} = \widetilde{\fu}$.
\end{enumerate}
\end{Prop}

\begin{proof}  We have $b = tu \in B^{(1, \nilp)}$ if and only if $b
\cdot (1, \nilp) = (t, tu \nilp) = (1, \nilp)$.  This holds if and only if $t = 1$ and
$u \in U^{\nilp}$, i.e. $b \in U^{\nilp} = B^{\nilp}$.  This proves (1).  

For (2), the assertion $\sigma^{-1}(B \cdot \nilp) = B \cdot (1, \nilp)$ follows from Lemma \ref{lem.orbits}.
Since $B \cdot \nilp$ is open in $\fu$, its inverse image $B \cdot (1, \nilp)$ is open in $\tilde{\fu}$.  Since
$\widetilde{\fu}$ is irreducible, it follows that
$\overline{B \cdot (1, \nilp)} = \widetilde{\fu}$.
\end{proof}

There is a map $B \cdot (1, \nilp) \to \cm$ defined by $b \cdot (1,\nilp) \mapsto b \cdot  \tilde{\nilp}$.
The equality of stabilizers $B^{(1, \nilp)}= B^{\tilde{\nilp}}$ implies that this map yields an isomorphism
$B \cdot (1, \nilp) \to B \cdot \tilde{\nilp}$.

\begin{Prop}  \label{p.normalization}
The map $B \cdot (1, \nilp) \to \cm $ extends to a map $\phi: \tilde{\fu} \rightarrow \cm$.
We have a $B$-equivariant commutative diagram
\begin{equation} \label{e.normvar}
 \begin{CD}
 \tilde{\fu} @>{\phi}>> \cm \\
 @V{\sigma}VV    @VV{\eta}V \\
 \fu @>>> \cn.
 \end{CD}
\end{equation}
The map $\phi$ induces a map
 $\psi: \tilde{\fu} \to \overline{ B \cdot \tilde{\nilp} }$, which
 is the normalization map.  We have $\psi^{-1}(B \cdot \tilde{\nilp}) = B \cdot (1,\nilp)$,
 and $\psi$ restricts to an isomorphism $B \cdot (1,\nilp) \to B \cdot \tilde{\nilp}$.
\end{Prop}

\begin{proof} 
To show that the map extends, it is enough to
show that the image of the pullback map $R(\cm) \to R(B \cdot (1,\nilp))$ lies 
in the subring $R(\tilde{\fu})$.  Consider the commutative diagram
\begin{equation} \label{e.normcommdiag}
\begin{CD}
R(\tilde{\fu})  @>>> R(B \cdot (1,\nilp)) @<<< R(\cm) \\
@AAA  @AAA    @AAA  \\
R(\fu) @>>> R(B \cdot \nilp) @<<< R(\cn).
\end{CD}
\end{equation}
Since
$R(\cm)$ is integral over $R(\cn)$, the image of $R(\cm)$ in $R(B \cdot (1,\nilp))$
is integral over the image of $R(\cn)$ in $R(B \cdot \nilp)$.
Since the map $B \cdot \nilp \to \cn$ extends to $\fu \to \cn$, we know that
the image of $R(\cn) \to R(B \cdot \nilp)$ lies in the subring $R(\fu)$.  We conclude
that the image of $R(\cm)$  in $R(B \cdot (1,\nilp))$ is integral over $R(\fu)$.  By Proposition \ref{p.bzcommute},
the integral closure of $R(\fu)$ in $\kappa(\widetilde{\fu}) = \kappa(B \cdot (1,\nilp))$ is $R(\tilde{\fu})$, so
 the image of $R(\cm)$ lies in $R(\tilde{\fu})$.  This proves the
first assertion of the proposition.  The commutative diagram \eqref{e.normvar} of affine varieties  follows
from the commutative diagram \eqref{e.normcommdiag} of rings of functions.  

Since $B \cdot (1,\nilp)$ is open in $\tilde{\fu}$, the image of $\tilde{\fu} \to \cm$
lies in the closure of the image of $B \cdot (1,\nilp)$, which is $\overline{ B \cdot \tilde{\nilp} }$.
Hence we obtain a birational map 
$\psi: \tilde{\fu} \to \overline{ B \cdot \tilde{\nilp} }$.  To show that this is the normalization map, 
since $R(\tilde{\u})$ is integrally closed  in $\kappa(\tilde{\fu})
= \kappa (\overline{B \cdot \tilde{\nilp} })$,
it suffices to show that $R(\tilde{\fu})$ is integral over 
$R(\overline{B \cdot \tilde{\nilp}})$.  To see this, observe that the map $\cm \to \cn$ takes
$\overline{B \cdot \tilde{\nilp}}$ to  $\fu$, so we have maps
$R(\fu) \to R(\overline{B \cdot \tilde{\nilp}}) \to R(\tilde{\fu})$.
Since $R(\tilde{\fu})$ is integral over $R(\fu)$,
the desired assertion follows.  

The assertion $\psi^{-1}(B \cdot \tilde{\nilp}) = B \cdot (1,\nilp)$ follows from Lemma \ref{lem.orbits}.
The fact that $\psi$ restricts to an isomorphism $B \cdot (1,\nilp) \to B \cdot \tilde{\nilp}$ follows from the
equality $B^{ (1,\nilp)} = B^{\tilde{\nilp}}$ of stabilizer groups.

Finally, all the maps in the commutative diagram \eqref{e.normvar} are $B$-equivariant.  Indeed, it was noted earlier that the map $\tilde{\fu} \to \fu$ is $B$-equivariant
and that the map $\cm \to \cn$ is $G$-equivariant.  The horizontal maps are the extensions of maps
from $B$-orbits; since the maps from $B$-orbits are $B$-equivariant by definition, all the maps
on rings of functions are $B$-equivariant, and therefore the extensions of the maps from the
$B$-orbits are $B$-equivariant as well.
\end{proof}

The variety $\overline{B \cdot \tilde{\nilp}}$ is not normal in general (see
Section \ref{s.nonnormal}).  

Define  $\tilde{\mu}: \tilde{\cm} \to \cm$ 
by $\tilde{\mu}([g,\xi]) = g \cdot \phi(\xi)$, where $\phi$ is as in Proposition \ref{p.normalization}.  

\begin{Thm} \label{t.genspringer}
The map $\tilde{\mu}: \tilde{\cm} \to \cm$ is proper and is an isomorphism over $\tilde{\co}^{pr}$.  There is a commutative
diagram
$$
\begin{CD} \label{e.commgenspringer}
\tilde{\cm} @>{\mnmap}>> \tilde{\cn} \\
@V{\mmap}VV         @VV{\mu}V \\
\cm @>{\eta}>> \cn,
\end{CD}
$$
where all maps are $G$-equivariant.  
The horizontal maps in this diagram are quotients by $Z$.
\end{Thm}

\begin{proof}
The map $\tilde{\mu}$ can be written as a composition:
$$
\tilde{\cm} = G \times^B \tilde{\fu} \to G \times^B \overline{ B \cdot \tilde{\nilp} } \to G \times^B \cm \stackrel{\cong}{\rightarrow} G/B \times \cm \to \cm.
$$
The first map is finite since it is induced by the normalization map $\tilde{\fu} \to \overline{ B \cdot \tilde{\nilp} }$, which
is finite.  The second map is a closed embedding.  The third map is the isomorphism which arises
because $\cm$ is a $G$-variety, not merely a $B$-variety; it is given by $[g,x] \mapsto (gB, g \cdot x)$.  The fourth map
is projection to the second factor.  Since all of these maps are proper, so is their composition $\tilde{\mu}$.

Since $\tilde{\fu} \to \overline{ B \cdot \tilde{\nilp} }$ is an isomorphism over the open orbit $B \cdot \tilde{\nilp}$,
the map $G \times^B \tilde{\fu} \to G \times^B \overline{ B \cdot \tilde{\nilp} }$ is an isomorphism
over the open subset $G \times^B B \cdot \tilde{\nilp}$.  The map $G \times^B B \cdot (1,\nilp) \to \cm$ takes
$G \times^B B \cdot (1,\nilp)$ isomorphically onto the orbit $G \cdot \tilde{\nilp} = \co^{pr}$.  It follows
that $\tilde{\mu}$ is an isomorphism over $\tilde{\co}^{pr}$.  Finally, the fact that the diagram \eqref{e.commgenspringer}
commutes follows from the commutativity of the diagram \eqref{e.normvar}, and $G$-equivariance follows
from the definitions of the maps.  

As noted in Section \ref{ss.cm}, $\cm/Z = \cn$.  To see that $\tilde{\cm}/Z = \tilde{\cn}$, 
note that the $Z$-action on $G\times^B \tilde{\u}$ is by
$$
z \cdot (g,u) = (zg, u) = (gz, u) = (g, z \cdot u).
$$
Since $\widetilde{\fu}/Z = \fu$ by Proposition \ref{p.bzcommute}, the fibers of the map $G\times^B \tilde{\u} \rightarrow
G\times^B \u$ are the orbits of $Z$, and the result follows.
\end{proof}

We have shown that $\tilde{\co}^{pr}$ is a subvariety of $\cm$; 
by Theorem \ref{t.genspringer}, we can also view $\tilde{\co}^{pr}$ as an open subvariety
of $\tilde{\cm}$.  Since $\tilde{\cm}$ is not in general smooth, $\tilde{\cm} \to \cm$ is not a resolution of singularities
of $\cm$.  However, the composition
 $\widehat{\cm} \stackrel{\hat{\eta}}{\rightarrow} \tilde{\cm} \stackrel{\tilde{\mu}}{\rightarrow} \cm$ 
is a resolution of the singularities of $\cm$.  Since $\tilde{\mu} \circ \widehat{\eta}$ is $G$-invariant
and an isomorphism over an open set, it is an isomorphism over $\widetilde{\co}^{pr}$.  

\begin{Cor} \label{c.regular}
We have isomorphisms
$$
R(\cm) \stackrel{\tilde{\mu}^*}{\rightarrow} R(\widetilde{\cm}) \stackrel{\widehat{\eta}^*}{\rightarrow} R(\widehat{\cm})
 \stackrel{i^*}{\rightarrow} R(\tilde{\co}^{pr}),
$$
where the first two maps are pullbacks, and the third map is restriction, i.e., pullback
via the inclusion $i:\tilde{\co}^{pr}\to \widehat{\cm}$.
\end{Cor}

\begin{proof} The maps $\widehat{\cm} \to \widetilde{\cm} \to \cm$ are isomorphisms
over $\widetilde{\co}^{pr}$, so $\widetilde{\co}^{pr}$ may be viewed as an open subvariety of
$\cm$, $\widetilde{\cm}$, and $\widehat{\cm}$.  Since each of these are irreducible
varieties, restriction to  $\widetilde{\co}^{pr}$ is injective on regular functions.
Thus, $i^*$ and $\widehat{\eta}^* \circ i^*$ are injective.
By Lemma \ref{lem.normalization}, $\tilde{\mu}^* \circ \widehat{\eta}^* \circ i^*$ is an isomorphism.
It follows that each of the maps $i^*$, $\widehat{\eta}^*$, and $\tilde{\mu}^*$ is an isomorphism.
\end{proof}

\section{Canonical sheaves and applications} \label{s.canonical}
In \cite{McG:89} and \cite{Hin:91}, the canonical sheaves on orbits and resolutions
are used in proving multiplicity formulas and proving that normalizations of orbit
closures are Gorenstein with rational singularities.  In this section we adapt these
arguments to $\cm$ using our resolution of singularities of $\cm$.  As observed
earlier, the map $\widetilde{\cm} \to \cm$ is not a resolution of singularities, since
$\widetilde{\cm}$ is locally a quotient of a smooth variety by a finite group.  Therefore, more background (e.g. an extension of the
Grauert-Riemenschneider theorem to the orbifold situation) would be required to apply
the arguments of \cite{McG:89} and \cite{Hin:91} to the map $\widetilde{\cm} \to \cm$.
In this paper, we circumvent this problem
by making use of the auxiliary variety $\widehat{\cm}$, which is smooth, and
the composition $\widehat{\cm} \to \widetilde{\cm} \to \cm$, which is a resolution of singularities.

The contents of the section
are as follows.  In Section \ref{ss.Gorenstein} we prove a result of Broer
\cite[Cor.~6.3]{Bro:98} that the variety $\cm$ is Gorenstein
with rational singularities (note that Broer's result applies in the setting
of covers of arbitrary nilpotent orbits).  The proof here is similar to the
proof given by Hinich \cite{Hin:91} for normalizations of orbit closures; we can use Hinich's argument because
we have a resolution of $\cm$ which also maps to $\widetilde{\cn}$.  We then turn to canonical and dualizing sheaves,
and describe the pushforward of the canonical
sheaf $\Omega_{\widehat{\cm}}$ on $\widehat{\cm}$ under the map $h = \widetilde{\eta} \circ \widehat{\eta}:\widehat{\cm} \to \tilde{\cn}$ (see Theorem \ref{thm.pushforwardmhat}), after
some preliminary results about these sheaves on toric varieties and on the
spaces $\widehat{\fu}$ and $\widetilde{\fu}$.  
Finally, in Section \ref{ss.mult},
we use this description to recover the formula for the $G$-module
decomposition of $R(\widetilde{\co}^{pr})$ from \cite{Gra:92}.  Note that the
formula from \cite{Gra:92} is different from the formula that arises from the techniques
of this paper; the equivalence of the formulas follows from Proposition \ref{p.conjugacy}.

\subsection{The Gorenstein property and rational singularities} \label{ss.Gorenstein}
Given an $n$-dimensional irreducible smooth variety $X$, recall that $\Omega^n_X$ denotes the sheaf of
top degree differential forms (i.e.,~differential $n$-forms) on $X$.  This is the sheaf of sections
of $\Exterior^n T^* X$.  As in Section \ref{s.preliminaries}, since we will only consider forms of top
degree, we write $\Exterior T^*_X = \Exterior^n T^*_X$.

We recall our varieties and maps:
$$
\begin{CD}
\widehat{\cm} @>\widehat{\eta}>> \tilde{\cm} @>{\mnmap}>> \tilde{\cn} \\
& & @V{\mmap}VV         @VV{\mu}V \\
& & \cm @>{\eta}>> \cn.
\end{CD}
$$
Write $\widehat{\mu}$ for the resolution of singularities  $\tilde{\mu} \circ \widehat{\eta}: \widehat{\cm} \to \cm$, 
and
$h = \widetilde{\eta} \circ \widehat{\eta}: \widehat{\cm} \to \widetilde{\cn}$.  The fact that $\widehat{\cm}$ maps to $\widetilde{\cn}$ as well
as to $\cm$ plays a significant role in the following proof.

\begin{Thm} \label{thm.gorenstein}
The variety $\cm$ is Gorenstein with rational singularities.
\end{Thm}

\begin{proof}
The map $\widehat{\mu}: \widehat{\cm} \to \cm$ 
is proper and birational, with
$\widehat{\cm}$ smooth and $\cm$ normal.  By a result of Hinich \cite[Lemma 2.3]{Hin:91}, 
if there is a map
$\varphi: \mathcal O_{\widehat{\cm}} \rightarrow \Omega_{\widehat{\cm}}$ such that
the induced map 
\begin{equation} \label{e.inducedmap}
\widehat{\mu}_*\mathcal O_{\widehat{\cm}} \rightarrow \widehat{\mu}_*\Omega_{\widehat{\cm}}
\end{equation}
is an isomorphism, then $\cm$ is Gorenstein with rational singularities.   Since
$\cm$ is an affine variety, to say that \eqref{e.inducedmap} is an isomorphism
amounts to saying that the map
\begin{equation} \label{e.inducedmap2}
H^0(\varphi):  H^0(\widehat{\cm}, \mathcal O_{\widehat{\cm}}) \rightarrow H^0(\widehat{\cm}, \Omega_{\widehat{\cm}})
\end{equation}
is an isomorphism.

The variety $\tilde{\cn}$ is a holomorphic symplectic variety (since it is isomorphic to the cotangent bundle
of $G/B$), and taking the top exterior power of the symplectic form gives a nowhere vanishing $G$-invariant
section $\Xi$ of $\Omega_{\tilde{\cn}}$.  The pullback $\widehat{\Xi} = h^*( \Xi)$ of this section
to $\widehat{\cm}$ is a $G$-invariant section of $\Omega_{\widehat{\cm}}$.  Since the restriction of $h$ to 
$\tilde{\co}^{pr}$ is a covering map, the section  $\widehat{\Xi}$ 
does not vanish at any point of the open set $\tilde{\co}^{pr}$.  Using $\widehat{\Xi}$, we define
the map $\varphi: \mathcal O_{\widehat{\cm}} \rightarrow \Omega_{\widehat{\cm}}$ as the sheaf map which
on any open set ${\mathcal U}$ takes $a \in  O_{\widehat{\cm}}({\mathcal U})$ to $a \widehat{\Xi}|_{U}$.
The induced map $H^0(\varphi)$ takes $a \in H^0(\widehat{\cm}, \mathcal O_{\widehat{\cm}}) $ to $a \  \widehat{\Xi}$.
This map is injective because the section $\Xi$ is not identically zero.  To show that the map $H^0(\varphi)$ is surjective,
suppose $\tau \in H^0(\widehat{\cm}, \Omega_{\widehat{\cm}})$.  Since $\Xi$ is nowhere zero on
$\tilde{\co}^{pr}$, the quotient $\tau/\widehat{\Xi}$ defines an element $r$ of $R(\tilde{\co}^{pr})$.  By Corollary
\ref{c.regular}, the element $r$
is the restriction of an element (also denoted by $r$) in $R(\widehat{\cm})$.
Hence $\tau = r  \  \widehat{\Xi} = H^0(\varphi)(r)$, so $H^0(\varphi)$ is surjective.  Hence $H^0(\varphi)$ is an isomorphism, proving the
theorem.
\end{proof}

As a consequence of the proof, we obtain the following.
\begin{Cor} \label{c.regform}
The map $R(\cm) \to H^0(\widehat{\cm}, \Omega_{\widehat{\cm}})$ taking $r$ to $r \ \widehat{\Xi}$ is a $G$-module isomorphism.
\end{Cor}

\begin{proof}
The map $r \mapsto r \ \Xi$ is the composition of the pullback map $R(\cm) \to R(\widehat{\cm})$ with the map
\eqref{e.inducedmap2}.  The pullback map is an isomorphism by Corollary \ref{c.regular}, and the map
\eqref{e.inducedmap2} is an isomorphism by the proof of Theorem \ref{thm.gorenstein}.  Hence
the map $r \mapsto r \ \Xi$ is an isomorphism; it is
$G$-equivariant because $\widehat{\Xi}$ is $G$-invariant.  
\end{proof}

\subsection{Canonical sheaves and toric varieties} \label{ss.canonicaltoric}
We recall some facts about toric varieties from \cite{Ful:93}.
An affine toric variety $X(\sigma)$ is smooth if and only if $\sigma$
is generated by part of a basis for $\P^{\vee}$.  A general toric
variety $X(\F)$ is smooth if and only if this holds for each cone
$\sigma$ in the fan $\F$.  

Suppose $\F'$ is a fan which is a subdivision of $\F$, i.e.,
$\F'$ and $\F$ have the same support, and every cone in $\F$
is a union of cones in $\F'$.  There is a $T$-equivariant
proper map $\pi: X(\F') \rightarrow X(\F)$.  Given any fan
$\F$, one can choose a subdivision $\F'$ such that $X(\F')$
is smooth, and then $\pi$ is a resolution of singularities.
In this situation, $\omega_{X(\F)} := \pi_*\Omega_{X(\F')}$ is independent of the subdivision
chosen, and $R^i \pi_* \Omega_{X(\F')} = 0$ for $i > 0$ (see \cite[Section 4.4]{Ful:93}).    On a
complete toric variety, $\omega_{X(\F)}$ is the dualizing sheaf.
For the affine toric variety $X(\sigma)$, we view $\omega_{X(\sigma)}$
simply as an $R(X(\sigma))$-module.  By \cite[Section~4.4]{Ful:93},
it can be viewed as the $R(X(\sigma))$-submodule of $R(T)$
spanned by the $e^{\mu}$ such that $\mu$ is positive on all nonzero
vectors in $\sigma$.

We now turn to the toric varieties $\toric$ and $\toric_{ad}$.  Since these varieties are affine,
we can view the sheaves $\omega_{\toric}$ and $\Omega_{\toric_{ad}}$ as (respectively)
 $R(\toric)$ and $R(\toric_{ad})$-modules.
The toric variety $\toric_{ad}$ is the vector space $\fu_2$ spanned by the $E_{\ga}$, where $\ga$ is simple, and
the functions $x_i = e^{\ga_i}$ give coordinates 
on $\toric_{ad}$ (see Section \ref{s.toric.g2}).  The sheaf $\Omega_{\toric_{ad}}$ is the
free $R(\toric_{ad})$ module generated by $dx_1 dx_2 \cdots dx_r$, so it is spanned by the elements
$x_1^{a_1} \cdots x_r^{a_r} dx_1 dx_2 \cdots dx_r$, where $a_i \geq 0$.  Each of these elements is a $T_{ad}$-weight
vector of weight $- \sum(a_i + 1) \ga_i = -(\xi + \sum a_i \ga_i)$.  (Recall that $x_i = e^{\ga_i}$ has weight $-\ga_i$;
the differential form $dx_1 dx_2 \cdots dx_r$ has weight $-\xi$, where
as in Section \ref{ss.prelimsemisimple}, $\xi$ is the sum of the simple roots.)  Thus, $\Omega_{\toric_{ad}}$ is isomorphic to the
 $R(\toric_{ad})$-submodule
of $R(T_{ad})$, spanned by the $e^{\mu}$, where $\mu = \sum (a_i + 1) \ga_i$.
We see that as a $R(\toric_{ad})$-module,
$\Omega_{\toric_{ad}} = R(\toric_{ad}) \otimes \C_{-\xi}$.  
Note that the cone $\sigma$ of $\V$ is generated by the elements of the basis $\{ v_i \}$ dual to the
basis of simple roots, and $\mu(v_i) = a_i + 1$.  That is, the elements $\mu$ are exactly those
elements of the lattice $\cq$ whose values are positive on any nonzero element of $\sigma$, recovering
the description of \cite[Section 4.4]{Ful:93}.

We can phrase this in the language of vector bundles.  Recall that if $V$ is a vector space, then $V_X$ denotes the vector bundle $V \times X \to X$.  The top exterior power of the
cotangent bundle of $\toric_{ad} = \fu_{2}$ is the bundle $\Exterior T^* \fu_2 = (\Exterior \fu_2^*)_{\toric_{ad}} =
( \C_{-\xi})_{\toric_{ad}}$, where the last equality holds since as a representation of
$T_{ad}$, $\fu_2^* = \C_{-\xi}$.  Therefore, as a $T$-equivariant vector bundle on $\toric_{ad}$,
$\Exterior T^* \toric_{ad} \cong (\C_{-\xi})_{\toric_{ad}}$.  This is a vector bundle which is trivial but not $T$-equivariantly
trivial.  The trivializing section has weight $-\xi$; it corresponds to the differential form $dx_1 dx_2 \cdots dx_r$.

Recall that in each coset of $\cp$ mod $\cq$ there is a unique element $\gl_R = \sum a_i \ga_i$ such that $0 \leq a_i < 1$.
Write $\gl_C = \xi - \gl_R$;  then $\gl_C = \sum b_i \ga_i$ is the unique element in the coset such that $0 < b_i \leq 1$.

Consider the maps
$$
\widehat{\toric} \stackrel{\widehat{\rho}}{\rightarrow} \toric \stackrel{\rho}{\rightarrow} \toric_{ad}.
$$

\begin{Prop} \label{p.toricpushforward}
As an $R(T_{ad})$-module with $T$-action,
\begin{equation} \label{e.toricpushforward1}
(\rho \circ \widehat{\rho})_* \Omega_{\widehat{\toric}} = \rho_*(\omega_{\toric}) = \bigoplus_{\gl_C} R(\toric_{ad}) \otimes \C_{- \gl_C} = \bigoplus_{\gl_R} \Omega_{\toric_{ad}} \otimes \C_{\gl_R}.
\end{equation}
Moreover, for $i>0$,
\begin{equation} \label{e.toricpushforward2}
R^i( \rho \circ \widehat{\rho})_* \Omega_{\widehat{\toric}} = R^i \rho_*(\omega_{\toric}) = 0.
\end{equation}
\end{Prop}

\begin{proof}
The equality $(\rho \circ \widehat{\rho})_* \Omega_{\widehat{\toric}} = \rho_*(\omega_{\toric})$ 
holds by definition of $\omega_{\toric}$. 
The sheaf $\rho_*(\omega_{\toric})$ corresponds to the $R(\toric_{ad})$-module given by the space of global
sections of $\omega_{\toric}$.  
By \cite[Section 4.4]{Ful:93}, this space corresponds to the set of $e^{\mu}$ for $\mu \in \cp$ which are
positive on any nonzero element of $\sigma$.  This means that if $\mu = \sum b_i \ga_i$, then each $b_i >0$.
Note that the $b_i$ are rational numbers, but not necessarily integers.  

Any $\mu = \sum c_i \ga_i$
in $\cp$, with each $c_i >0$, can be written uniquely as $\mu = \gl_C + \nu$ for some $\gl_C$, where $\nu = \sum d_i \ga_i$
and each $d_i$ is a nonnegative integer.  That is, $e^{\mu} = e^{\nu} \cdot e^{\gl_C}$.  
We see
that $\rho_*(\omega_{\toric})$, viewed as an $R(\toric_{ad})$-module, is equal to the $R(T_{ad})$-submodule of $R(T)$ given by
$\oplus R(T_{ad}) \cdot e^{\gl_C}$, where the sum is over all $\gl_C$.  
Since $e^{\gl_C}$ is a $T$-weight vector of weight $-\gl_C$, we see that
as an $R(T_{ad})$-module with $T$-action,
$\rho_*(\omega_{\toric}) = \oplus R(\toric_{ad}) \otimes \C_{- \gl_C}$, proving
the second equality of \eqref{e.toricpushforward1}.  Since 
$\Omega_{\toric_{ad}} = R(\toric_{ad}) \otimes \C_{- \xi}$, we have
$$
\rho_*(\omega_{\toric}) = \bigoplus \Omega_{\toric_{ad}} \otimes \C_{\xi - \gl_C} = \bigoplus \Omega_{\toric_{ad}} \otimes \C_{\gl_R}.
$$
proving the third equality of \eqref{e.toricpushforward1}.  

To prove \eqref{e.toricpushforward2}, consider
the spectral sequence $R^i\rho_* \circ R^j\widehat{\rho}_*
\Rightarrow R^{i+j}(\rho \circ \widehat{\rho})_*$.  Since $\rho$ is a finite map, for any coherent sheaf $\cf$ on $\fu$, we have $R^i \rho_*(\cf) = 0$
for $i>0$.  By \cite{Ful:93}, $R^i \widehat{\rho}_*\Omega_{\widehat{\toric}} = 0$ for $i>0$.
We conclude that for $i>0$, we have  $R^i(\rho \circ \widehat{\rho})_* \Omega_{\widehat{\toric}}= 0$, as desired. 
\end{proof}

\subsection{Canonical sheaves on $\widehat{\fu}$ and $\widetilde{\fu}$}
Recall our maps
$\widehat{\fu} \stackrel{\widehat{\sigma}}{\longrightarrow} \widetilde{\fu}  \stackrel{\sigma}{\longrightarrow} \fu$, with
$\gamma = \sigma \circ \widehat{\sigma}$.  As above, we let $\widehat{p}_1$, $\tilde{p}_1$, and $p_1$ denote the projections of $\widehat{\fu}$, $\widetilde{\fu}$, and $\fu$ onto $\widehat{\toric}$, $\toric$, and $\toric_{ad}$, respectively.
We define $\omega_{\widetilde{\fu}} = \widehat{\sigma}_* \Omega_{\widehat{\fu}}$.
 In this section we describe $\sigma_* \omega_{\widetilde{\fu}} = \gamma_*  \Omega_{\widehat{\fu}}$.
 The result is analogous to Proposition \ref{p.toricpushforward}; the main issue in the proof
 is showing that the description obtained is $B$-equivariant ($T$-equivariance is clear).

As a variety, 
\begin{equation} \label{e.fu}
\widehat{\fu} = \widehat{\toric} \times \fu_{\geq 4}.
\end{equation}
  Since the tangent bundle
$\fu_{\geq 4}$ is the trivial bundle $(\fu_{\geq 4})_{\fu_{\geq 4}}$, we have an identification
\begin{equation} \label{e.futangent}
T \widehat{\fu} = \widehat{p}_1^* T \widehat{\toric} \oplus (\fu_{\geq 4})_{\widehat{\fu}}
\end{equation}
 of vector
bundles on $\widehat{\fu}$.  Moreover, this identification is $T$-equivariant, since the decomposition
\eqref{e.fu} is $T$-equivariant.  However, we cannot assert that the identifcation is $B$-equivariant, since
\eqref{e.fu} is not $B$-equivariant.  Nevertheless, we have the following.

\begin{Lem} \label{lem.bequiv}
There are $B$-equivariant exact sequences of vector bundles on $\widehat{\fu}$ and $\fu$, respectively:
$$
0 \to  (\fu_{\geq 4})_{\widehat{\fu}} \to T \widehat{\fu} \to  \widehat{p}_1^* T \widehat{\toric}
\to 0,
$$
and
$$
0 \to  (\fu_{\geq 4})_{\fu} \to T \fu \to  p_1^* T \toric_{ad} \to 0.
$$
\end{Lem}
\begin{proof}
We prove the lemma for the first sequence; the proof for the second is similar.  The existence of the
first exact sequence follows from the identification \eqref{e.futangent}.  We must prove that the maps
involved are $B$-equivariant.  Since the map $\widehat{p}_1: \widehat{\fu} \to \widehat{\toric}$ is 
$B$-equivariant, so is the induced map $T \widehat{\fu} \to  \widehat{p}_1^* T \widehat{\toric}$.  The
map $(\fu_{\geq 4})_{\widehat{\fu}} \to T \widehat{\fu}$ is $T$-equivariant, as follows from the $T$-invariance of
\eqref{e.futangent}, so it suffices to verify that the map is $U$-equivariant.  Suppose $(\widehat{x},q)
\in \widehat{\toric} \times \fu_{\geq 4} \cong \widehat{\fu}$.  Let $x = \rho \circ \widehat{\rho} \circ \widehat{p}_1(\widehat{x})
\in \toric_{ad}$.  Let $x'$ denote the element $x$, but viewed as an element of $\fu$ via
our identification of $\fu_2$ as a subspace of $\fu$.  Note that if $u \in U$, while we have
$u x = x$, we may not have $u x' = x'$.  We have
$$
u \cdot (\widehat{x},q) = (\widehat{x}, ux'-x' + uq).
$$

Let $\xi \in \fu_{\geq 4}$, and let $\xi_{(\widehat{x},q)}$
denote the corresponding element in the fiber $(\widehat{\fu}_{\geq 4})_{{\fu}, (\widehat{x},q)}$
The map $(\fu_{\geq 4})_{\widehat{\fu}} \to T \widehat{\fu}$ takes $\xi_{(\widehat{x},q)}$ to the vector
$\xi^+_{(\widehat{x},q)}$ in $T_{(\widehat{x},q)} \widehat{\fu}$.  Here $\xi^+_{(\widehat{x},q)}$ is the
tangent vector which, when applied to a test function $\varphi$, yields
$$
\xi^+_{(\widehat{x},q)} (\varphi) = \frac{d}{dt} \varphi(\widehat{x}, q + t \xi)|_{t=0}.
$$
To prove $U$-equivariance of the map $(\fu_{\geq 4})_{\widehat{\fu}} \to T \widehat{\fu}$, it suffices to show that if
 $u \in U$, then 
$$
u_* (\xi^+_{(\widehat{x},q)}) = (u \cdot \xi)^+_{u \cdot (\widehat{x},q)}.
$$
This follows from a direct calculation:
\begin{eqnarray*}
u_* (\xi^+_{(\widehat{x},q)}) (\varphi) & = &  \xi^+_{(\widehat{x},q)} (\varphi \circ u) 
=  \frac{d}{dt} \varphi \circ u(\widehat{x}, q + t \xi)|_{t=0} \\
& = & \frac{d}{dt} \varphi (\widehat{x}, ux' - x' + uq + t u \cdot \xi)|_{t=0} \\
& = & (u \cdot \xi)^+_{u \cdot (\widehat{x},q)}(\varphi).
\end{eqnarray*}
This completes the proof.
\end{proof}

The analogue of Proposition \ref{p.toricpushforward} is the following.

\begin{Prop} \label{p.upushforward}
As $B$-equivariant sheaves on $\fu$,
\begin{equation} \label{e.upushforward1}
\gamma_* \Omega_{\widehat{\fu}} = \sigma_*(\omega_{\widetilde{\fu}})  = \bigoplus_{\gl_R} \Omega_{\fu} \otimes \C_{\gl_R}.
\end{equation}
Also, for $i>0$,
\begin{equation} \label{e.upushforward2}
R^i \gamma_*  \Omega_{\widehat{\fu}} = 0.
\end{equation}
\end{Prop}

\begin{proof}
The first equality of \eqref{e.upushforward1} follows from the definition of $\omega_{\widetilde{\fu}}$, so we show
$\gamma_* \Omega_{\widehat{\fu}} = \bigoplus_{\gl_R} \Omega_{\fu} \otimes \C_{\gl_R}$.  The exact
sequences of Lemma \ref{lem.bequiv} imply that as $B$-equivariant sheaves,
$ \Omega_{\widehat{\fu}} = \widehat{p}_1^*  \Omega_{\widehat{\toric}} \otimes \Exterior \fu_{\geq 4}^*$
and $ \Omega_{\fu} = p_1^*  \Omega_{\toric_{ad}} \otimes \Exterior \fu_{\geq 4}^*$.
 Therefore, with maps as in \eqref{e.fiber}, we have
 \begin{eqnarray*}
 \gamma_*  \Omega_{\widehat{\fu}}  & = & \gamma_*  \widehat{p}_1^*  \Omega_{\widehat{\toric}} \otimes \Exterior \fu_{\geq 4}^*   = \sigma_* \widehat{\sigma}_* \widehat{p}_1^*  \Omega_{\widehat{\toric}} \otimes \Exterior \fu_{\geq 4}^* \\
 & = & \sigma_* \widetilde{p}_1^* \widehat{\rho}_* \Omega_{\widehat{\toric}} \otimes \Exterior \fu_{\geq 4}^* 
=  \sigma_* \widetilde{p}_1^* \omega_{\widetilde{\toric}} \otimes \Exterior \fu_{\geq 4}^*  \\
& = & p_1^* \rho_* \omega_{\widetilde{\toric}} \otimes \Exterior \fu_{\geq 4}^* \\
& = & \bigoplus p_1^* \Omega_{\toric_{ad}} \otimes \Exterior \fu_{\geq 4}^* \otimes \C_{\gl_R} =  \bigoplus_{\gl_R} \Omega_{\fu} \otimes \C_{\gl_R}.
 \end{eqnarray*}
This proves \eqref{e.upushforward1}.  To prove the vanishing result \eqref{e.upushforward2}, it suffices
to show that $R^i \gamma_* \widehat{p}_1^* \Omega_{\widehat{\toric}} = 0$ for $i>0$.  Since $\gamma = \sigma \circ \widehat{\sigma}$,
and $\sigma$ is a finite map, it suffices (as in the proof of Proposition \ref{p.toricpushforward}) to show that
$R^i \widehat{\sigma}_*( \widehat{p}_1^* \Omega_{\widehat{\toric}}) = 0$ for $i>0$.
In the left fiber square in \eqref{e.fiber}, the maps $\widehat{p}_1$ and $\widetilde{p}_1$ are flat,
so $R^i \widehat{\sigma}_*( \widehat{p}_1^* \Omega_{\widehat{\toric}} )
\cong \widetilde{p}_1^* R^i \widehat{\rho}_* ( \Omega_{\widehat{\toric}} )$.  As noted above,
$R^i \widehat{\rho}_* ( \Omega_{\widehat{\toric}} ) = 0$ for $i>0$ by \cite{Ful:93}; \eqref{e.upushforward2} follows.
\end{proof}

Recall our commutative diagram
$$
\begin{CD}
\widehat{\cm} @>{\widehat{\eta}}>> \tilde{\cm} @>{\mnmap}>> \widetilde{\cn} \\
& & @V{\mmap}VV         @VV{\mu}V \\
& & \cm @>{\eta}>> \cn,
\end{CD}
$$
as well as the maps $h = \widetilde{\eta} \circ \widehat{\eta}: \widehat{\cm} \to \widetilde{\cn} $ and 
$\widehat{\mu} = \widetilde{\mu} \circ \widehat{\eta}: \widehat{\cm} \to \cm$.

\begin{Thm} \label{thm.pushforwardmhat}
Let $\pi: \widetilde{\cn} \to G/B$ denote the projection, and $\cl_{\gl}$ the sheaf of sections of the line
bundle $G \times^B \C_{\gl}$ on $G/B$.
As $G$-equivariant sheaves on $\widetilde{\cn}$, we have
$$
h_* \Omega_{\widehat{\cm}} = \bigoplus_{\gl_R} \pi^* \cl_{\gl_R}
$$
Moreover,
$$
R^i h_*  \Omega_{\widehat{\cm}} = 0
$$
for $i>0$.
\end{Thm}

\begin{proof}
By definition, $\widehat{\cm} = G \times^B \widehat{\fu}$,
$\widetilde{\cm} = G \times^B \widetilde{\fu}$ and $\widetilde{\cn} = G \times^B \fu$.
Let $\widehat{\pi}$,  $\widetilde{\pi}$ and $\pi$, respectively, denote the projections of these spaces to $G/B$.
By Lemma \ref{lem.mixedcanonical}, we have
$$
\Omega_{\widehat{\cm} } = \CInd_B^G \Omega_{\widehat{\fu}} \otimes \widehat{\pi}^* \Omega_{G/B}.
$$
Since $\widehat{\pi} = \pi \circ h$, we have
\begin{equation} \label{e.pushforwardmhat}
h_* \Omega_{\widehat{\cm} } =
h_* \Big( \CInd_B^G \Omega_{\widehat{\fu}} \otimes  h^* \pi^* \Omega_{G/B} \Big)
= h_* \Big( \CInd_B^G \Omega_{\widehat{\fu}} \Big) \otimes \pi^* \Omega_{G/B}.
\end{equation}
The map $h$ is $\CInd_B^G \gamma$, where $\gamma = \sigma \circ \widehat{\sigma}: \widehat{\fu} \to \fu$.  
By the compatibility of $\CInd_B^G$ with direct image, the right hand side is 
$\CInd_B^G ( \gamma_* \Omega_{\widehat{\fu}} ) \otimes \pi^* \Omega_{G/B}$.  By Proposition \ref{p.upushforward}
and Lemma \ref{lem.mixedcanonical},
this equals
$$
\bigoplus_{\gl_R} \CInd_B^G (\Omega_{\fu} \otimes \C_{\gl_R}) \otimes \pi^* \Omega_{G/B}
= \bigoplus_{\gl_R} \CInd_B^G (\Omega_{\fu})  \otimes \pi^* \cl_{\gl_R} \otimes \pi^* \Omega_{G/B} = \bigoplus_{\gl_R} \Omega_{\widetilde{\cn}} \otimes \pi^*\cl_{\gl_R}.
$$
Since $\widetilde{\cn}$  is isomorphic to the cotangent bundle of $G/B$, it has a $G$-invariant
holomorphic symplectic form, whose top exterior power is a $G$-invariant nowhere vanishing
section of $\Omega_{\widetilde{\cn}}$.  Hence, as $G$-equivariant sheaves,
$\Omega_{\widetilde{\cn}} \cong \co_{\widetilde{\cn}}$.
This proves the first equation of the corollary.  Similarly, arguing as in \eqref{e.pushforwardmhat}, we have
$$
R^i h_*  \Omega_{\widehat{\cm}} = \CInd_B^G(R^i \gamma_*  \Omega_{\widehat{\fu}} )  \otimes \pi^* \Omega_{G/B}.
$$
For $i>0$, $R^i \gamma_*  \Omega_{\widehat{\fu}} =0$ by Proposition \ref{p.upushforward}, so 
$R^i h_*  \Omega_{\widehat{\cm}} = 0$, as desired.
\end{proof}

\subsection{Multiplicity formulas} \label{ss.mult}
The main result of this section is a proof of a formula from \cite{Gra:92} for the
$G$-module decomposition of $R(\cm) = R(\widetilde{\co}^{pr})$.  The argument here is adapted
from the argument given in \cite{McG:89}, where a formula is given for the ring of
functions on an arbitrary nilpotent orbit.  This proof, which is different from the argument in \cite{Gra:92}, is possible because
we have a resolution of $\cm$ which also maps to $\widetilde{\cn}$.
The formula
arising from the arguments of this paper is given in terms of the weights $\gl_R$, while the
formula in \cite{Gra:92} is in terms of the weights $\gl_{dom}$.  The fact that the two formulas
are equivalent follows from  
Proposition \ref{p.conjugacy}, which shows that the
weights $\lambda_{dom}$ and $\lambda_R$ (belonging to the same coset
of $\P$ mod $\Q$) are conjugate by the Weyl group $W$.

By definition, for any weight $\mu \in \cp$, $\Ind_T^G (\C_{\mu})$ is the space of global sections
of the vector bundle $G \times^T \C_{\mu} \to G/T$ on the affine variety $G/T$.  As a representation of $G$, 
$\Ind_T^G (\C_{\mu})$ is a direct sum of finite dimensional irreducible representations of $G$.  By Frobenius
reciprocity, the multiplicity of an irreducible representation $V$ in $\Ind_T^G (\C_{\mu})$ equals the dimension
of the $\mu$-weight space of $V$.  Since $W$-conjugate weights occur with the same multiplicity
in $V$ (see \cite{Hum:72}), as representations of $G$, $\Ind_T^G (\C_{\mu}) = \Ind_T^G (\C_{w \mu})$ for
any $w \in W$.

Recall that in each coset of $\P$ mod $\Q$ we have defined
elements $\lambda_R$ and $\lambda_{dom}$.  For the identity coset $\Q$, $\gl_R = \gl_{dom} = 0$ (the convention
of \cite{Gra:92} was that $\gl = 0$ was not considered as a weight of the form $\gl_{dom}$, but listed separately).

\begin{Thm}  \label{t.mult}
As $G$-modules,
$$
R(\widetilde{\co}^{pr}) = R(\cm) =  \bigoplus_{\lambda_R}
\Ind^G_T(\C_{\lambda_R}).
$$
\end{Thm}

\begin{proof}  
The idea of the proof is to use the fact (see \cite[Lemma 2.1]{McG:89}) that as a $G$-module, for any weight $\mu$, we have
\begin{equation} \label{e.mcgovern}
\chi(\widetilde{\cn}, \pi^* \cl_{\mu}) = \Ind_T^G (\C_{\mu}),
\end{equation}
where $\pi: \widetilde{\cn} \to G/B$ is the projection, and 
$\chi(\widetilde{\cn}, \cf) = \sum (-1)^i H^i(\widetilde{\cn}, \cf)$ denotes the Euler
characteristic.  To apply this fact, we will express $R(\cm)$ in terms of global sections
on $\widetilde{\cn}$, and then use cohomology vanishing to identify the space of global sections
with the Euler characteristic.  

The argument is as follows.  By Corollary \ref{c.regform}, $R(\cm) \cong H^0(\widehat{\cm}, \Omega_{\widehat{\cm}})$.
Consider the composition
$$
\widehat{\cm} \stackrel{h}{\longrightarrow} \widetilde{\cn} \stackrel{\mu}{\longrightarrow} \cn.
$$
The variety $\cn$ is affine, and for any $i$, the sheaf $R^i (\mu \circ h)_* (\Omega_{\widehat{\cm}})$ corresponds
to the $R(\cn)$-module $H^i(\widehat{\cm}, \Omega_{\widehat{\cm}})$.  In particular,
$R(\cm)$ is the $R(\cn)$-module corresponding to the sheaf
$$
(\mu \circ h)_* (\Omega_{\widehat{\cm}}) = \mu_* \circ h_*  (\Omega_{\widehat{\cm}}) =
\mu_*( \oplus_{\gl_R}  \pi^* \cl_{\gl_R}),
$$
where the last equality follows from
Theorem \ref{thm.pushforwardmhat}.  Since $\cn$ is affine, this sheaf corresponds to the
$R(\cn)$-module $H^0(\widetilde{\cn}, \oplus_{\gl_R}  \pi^* \cl_{\gl_R})
=  \oplus_{\gl_R} H^0(\widetilde{\cn},  \pi^* \cl_{\gl_R})$.  In light of the equality \eqref{e.mcgovern}, to complete the proof, it suffices
to show that for $i>0$,  $H^i(\widetilde{\cn},  \pi^* \cl_{\gl_R}) = 0$.

Recall that $\widehat{\mu} = \widetilde{\mu} \circ \widehat{\eta}: \widehat{\cm} \to \cm$.  We have
$\mu \circ h = \eta \circ \widehat{\mu}$.  Therefore,
$$
R^i  (\mu \circ h)_*( \Omega_{\widehat{\cm}}) = R^i  ( \eta \circ \widehat{\mu})_* (\Omega_{\widehat{\cm}}) .
$$
There is a spectral sequence
$$
R^i \eta_* \circ R^j \widehat{\mu}_* ( \Omega_{\widehat{\cm}}) \Rightarrow R^{i+j}  ( \eta \circ \widehat{\mu})_* (\Omega_{\widehat{\cm}}).
$$
Since $\eta$ is a finite map, the higher direct images $R^i \eta_*$ vanish.  Since the map $\widehat{\mu}$ is proper
and birational, the Grauert-Riemenschneider theorem implies
that $R^j  \widehat{\mu}_* (\Omega_{\widehat{\cm}}) = 0$ for $j>0$.  We conclude
that for $i>0$, $R^i ( \eta \circ \widehat{\mu})_* (\Omega_{\widehat{\cm}}) = 0$, so 
$R^i  (\mu \circ h)_*( \Omega_{\widehat{\cm}}) = 0$.  There is a spectral sequence
$$
R^i \mu_* \circ R^j h_* ( \Omega_{\widehat{\cm}}) \Rightarrow R^{i+j}  (\mu \circ h)_* (\Omega_{\widehat{\cm}}).
$$
By Theorem \ref{thm.pushforwardmhat}, $R^j h_*  \Omega_{\widehat{\cm}} = 0$ for $j>0$.  Hence, for  $i>0$, 
$(R^i \mu_*) h_* ( \Omega_{\widehat{\cm}}) = 0$.  
Since $\cn$ is affine, the sheaf 
$(R^i \mu_*) h_* ( \Omega_{\widehat{\cm}})$ corresponds to the $R(\cn)$-module
$H^i(\widetilde{\cn}, h_* ( \Omega_{\widehat{\cm}}))$.  Also, by Theorem \ref{thm.pushforwardmhat},
$h_* \Omega_{\widehat{\cm}} = \oplus_{\gl_R}  \pi^* \cl_{\gl_R}$.
We conclude that for all $\gl_R$ and for all $i>0$, we have
$H^i(\widetilde{\cn},  \pi^* \cl_{\gl_R}) = 0$.  This completes the proof.
\end{proof}

Observe that the proof of Theorem \ref{t.mult} yielded the fact that  
for any $\gl_R$ and any $i>0$, we have
\begin{equation} \label{e.vanish}
H^i(\widetilde{\cn},  \pi^* \cl_{\gl_R}) = 0.
\end{equation}
Since we have chosen our Borel subalgebra to correspond to positive weight spaces (which
is the opposite of the convention of \cite{Gra:92}), general principles would suggest a vanishing theorem corresponding
to negative  weights.  The equation above shows that the vanishing also holds for weights which are slightly nonnegative.
In fact, more is true: one can show using a vanishing theorem of Hesselink (see \cite{Hes:76}) that this vanishing holds
for any $W$-conjugate of $\gl_{dom}$ (see \cite[Theorem 1.3]{Gra:92}).

The equivalence of Theorem \ref{t.mult} with the multiplicity result of \cite{Gra:92} is a consequence of the following
proposition.

\begin{Prop} \label{p.conjugacy}  The weights
$\lambda_{dom}$ and $\lambda_R$ (in a fixed coset of $\P$ mod
$\Q$) are $\W$-conjugate.
\end{Prop}

\begin{proof} This can be proved by
explicit calculation.  For example, consider type $A_{n-1} = \mathfrak{sl}_n$.  
In this case, $\t$ can be identified as the subspace of $\C^n$ defined by the
equation $x_1 + \cdots + x_n = 0$; the inner product
$( (x_1, \ldots, x_n), (y_1, \ldots, y_n) ) = \sum x_i y_i$
gives an identification $\t \simeq \t^*$.  We take as simple
roots $\ga_i = \epsilon_i - \epsilon_{i+1}$.  The fundamental
dominant weights $\lambda_1, \ldots, \lambda_{n-1}$ are all
minuscule; in coordinates,
$$
\lambda_i = \frac{1}{n}(n-i, \ldots, n-i, -i, \ldots, -i),
$$
where the first $i$ entries are equal to $n-i$.  Suppose $\lambda_{dom} = \lambda_i$,
and write
$$
\lambda_R = (c_1, \ldots, c_n) = \sum a_k \ga_k.
$$
Then each $c_k$ satisfies $n c_k \equiv -i$ (mod $n$), and
$\sum c_k = 0$.  The conditions $0 \leq a_k < 1$
imply that $-n < n c_k < n$ for all $k$; therefore each
$n c_k$ is either $-i$ or $n-i$.  The condition $\sum c_k = 0$
implies that exactly $i$ of the $n c_k$ must equal $n - i$.
It follows that the coordinates $(c_1, \ldots, c_n)$ are
a permutation of $\frac{1}{n}(n-i, \ldots, n-i, -i, \ldots, -i)$, so
$\lambda_R$ is $W$-conjugate to $\lambda_{dom}$.  This
proves the result for $A_{n-1}$; proofs for the other classical groups
are similar but easier.

The exceptional groups with nonzero minimal dominant weights are
types $E_6$ and $E_7$.  For $E_6$, using the notation of \cite{Hum:72},
the only minuscule weight is 
$$
\lambda_1 = \frac{1}{3} (4 \ga_1 + 3 \ga_2 + 5 \ga_3
+ 6 \ga_4 + 4 \ga_5 + 2 \ga_6) .
$$
Let $\lambda_{dom} = \lambda_1$; then 
$$
\lambda_R = \frac{1}{3} ( \ga_1 + 2 \ga_3 +  \ga_5 + 2 \ga_6).
$$
If the inner product $(\mu, \ga_i)$ is $1$, then the corresponding
simple reflection $s_i$ takes $\mu$ to $\mu - \ga_i$.  Using this fact
repeatedly, we see that $w = s_4 s_5 s_2 s_4 s_3 s_1$ takes
$\lambda_{dom}$ to $\lambda_R$.  The proof for $E_7$ is
similar. \end{proof}

Using this, we recover the multiplicity formula of \cite{Gra:92}.

\begin{Cor} \label{c.graham}
As $G$-modules, 
$$
R(\tilde{\co}^{pr}) = R(\cm) = \bigoplus_{\lambda_{dom}}
\Ind^G_T(\C_{\lambda_{dom}}).
$$
\end{Cor}

\begin{proof}
If $\gl_{R}$ and $\gl_{dom}$ are in the same coset of $\cp$ mod $\cq$, then they are $W$-conjugate
by Proposition \ref{p.conjugacy}, so $\Ind_T^G(\C_{\gl_R}) = \Ind_T^G(\C_{\gl_{dom}})$.
\end{proof}

\section{Nonnormality of a $B$-orbit closure} \label{s.nonnormal}
The methods of this paper, together with a result of \cite{Gra:92},
allow us to show that 
the variety $\cm$ differs from $\cn$ in that the subvariety
$\overline{B \cdot \widetilde{\nilp}}$ of $\cm$ is not normal in general.  In contrast, 
$\overline{B \cdot \nilp} =
\u$ is isomorphic to affine space, so it is a normal subvariety of $\cn$.  

By Proposition \ref{p.normalization}, there is a map $\phi: \widetilde{\fu} \to \cm$ which yields
a map $\psi: \widetilde{\fu} \to \overline{B \cdot \tilde{\nilp}}$; the map $\psi$ is the
normalization map.  Via pullback by $\psi$, we can identify 
$R(\overline{B \cdot \tilde{\nilp}})$ with a subring of $R(\tilde{\u})$.  The two rings are equal if and only if
$\overline{B \cdot \tilde{\nilp}}$ is normal.  We will describe
 $R(\overline{B\cdot \tilde{\nilp}})$ as a subring of $R(\tilde{\u})$, and thus determine when 
 $\overline{B\cdot \tilde{\nilp}}$ is normal.  In particular, if $\fg$ is simple, this occurs
 only when $\fg$ is of type $A_1$ or $A_2$.

As a variety, $\widetilde{\fu} = \toric \times \fu_{\geq 4}$, so
we have an identification $R(\widetilde{\fu}) = R(\toric) \otimes R(\fu_{\geq 4})$.
The ring $R(\toric_{ad})$ is the subring of $R(T)$
generated by the $e^{\ga}$ as $\ga$ runs over the simple roots.
The ring $R(\toric)$ is the subring of $R(T)$ generated by the $e^{\gl_R}$
and $R(\toric_{ad})$.  Let $S$ be the subring of $R(\toric)$ which
is generated by the $e^{\gl_{dom}}$ and $R(\toric_{ad})$.  The main result of this section is the following.

\begin{Thm} \label{thm.bclosure}
With notation as above, $R(\overline{B \cdot \tilde{\nilp}})$ is the subring 
$S \otimes R(\fu_{\geq 4})$
of $R(\widetilde{\fu}) = R(\toric) \otimes R(\fu_{\geq 4})$.
\end{Thm}

\begin{proof}
We have
$$
R(\cm) \to R(\overline{B \cdot \tilde{\nilp}}) \to R(\widetilde{\fu}),
$$
so the image of $R(\overline{B \cdot \tilde{\nilp}})$ in $R(\widetilde{\fu})$ is the same as
the image of $R(\cm)$.  Write $A$ for the image of $R(\cm)$ in $R(\widetilde{\fu})$; we want to show
that $A = S \otimes R(\fu_{\geq 4})$.  As a ring, $A$ is generated by the images of generators of $R(\cm)$.
The commutative diagram \eqref{e.normvar} 
yields a commutative diagram of rings of regular functions:
$$
\begin{CD}
R(\widetilde{\fu}) @<<< R(\cm )\\
@VVV    @VVV \\
R(\fu) @<<< R(\cn)
\end{CD}
$$
Each minuscule representation occurs exactly once in $R(\cm)$, and as a module for $R(\cn)$,
$R(\cm)$ is generated by $1$ and the minuscule representations in $R(\cm)$ (see \cite[Cor.~3.4]{Gra:92}).
The image of $R(\cn)$ in $R(\widetilde{\fu})$ under the compositions in the above commutative
diagram is the image of $R(\fu) = R(\toric_{ad}) \otimes R(\fu_{\geq 4})$ in
 $R(\widetilde{\fu})$.  To complete the proof we  consider the minuscule representations in $R(\cm)$ and
 describe their images in $R(\widetilde{\fu})$.   
 
 Suppose
 $V$ is minuscule with highest weight $\gl = \gl_{dom}$.  Each weight of $V$ occurs with multiplicity $1$;
choose a basis $\{ v_{\mu} \}$ of 
 $V$, where $v_{\mu}$ lies in the $\mu$-weight space of $V$.  Let $w_0$ be the longest element of the Weyl group $W$.  
We claim the following.  The weight vector 
 $v_{w_0 \gl}$ maps to a multiple of the function $e^{-w_0 \gl} = e^{-w_0 \gl} \otimes 1 \in R(\widetilde{\fu})$ (which
 is a weight vector of weight $w_0 \gl$);
 if $\mu \neq w_0 \gl$, $v_{\mu}$ maps to $0$ in $R(\widetilde{\fu})$.
  The theorem
 is a consequence of the claim and the preceding discussion, noting that $-w_0 \gl$ is itself a minimal dominant weight.
 
Consider the commutative diagram
\begin{equation} \label{e.diagramregular}
\begin{CD}
R(\cm) @>{\cong}>> R(G/G^{\tilde{\nilp}}) \\
@VVV   @VVV \\
R(\overline{B \cdot \tilde{\nilp}}) @>>> \hspace{.5in} R(B/B^{\tilde{\nilp}})  \supset R(\widetilde{\fu}).
\end{CD}
\end{equation}
The top arrow is induced by the map $G/G^{\tilde{\nilp}}  = \widetilde{\co}^{pr} \to G \cdot \widetilde{\nilp}
\hookrightarrow \cm$.  The bottom map is the restriction map to $B \cdot \tilde{\nilp} \cong B/B^{\tilde{\nilp}}$; the
inclusion $R(\widetilde{\fu}) \subset R(B/B^{\tilde{\nilp}})$ follows from the identifications
$B/B^{\tilde{\nilp}} \cong B/B^{(1,\nilp)} \cong B \cdot (1,\nilp) \subset \widetilde{\fu}$ (recall that
$B^{\tilde{\nilp} }= B^{(1,\nilp)}$ by Proposition \ref{prop.stabilizer}).
We know that under these identifications, the image of the map  $R(G/G^{\tilde{\nilp}}) \to R(B/B^{\tilde{\nilp}})$ lies in
the subset $R(\widetilde{\fu})$ 
(cf.~Proposition \ref{p.normalization}).  The function $e^{\mu} \in R(\widetilde{\fu})$
corresponds to the function $r \in R(B/B^{\tilde{\nilp}})$ satisfying  $r(b) = e^{\mu}(b)$, where as usual
we extend the function $e^{\mu}$ from $T$ to $B$ by requiring $e^{\mu}(tu) = e^{\mu}(t)$ for $t \in T$, $u \in U$.

In light of the commutative diagram \eqref{e.diagramregular}, to prove the claim, it suffices to show
that (in the setting of the claim) if $r \in R(G/G^{\tilde{\nilp}}) $ corresponds to $v_{\mu} \in V$, then $r(b) = 0$ if
$\mu \neq w_0 \gl$, and $r(b) = e^{- w_0 \gl}(b)$ if $\mu = w_0 \gl$.

Let $\{ f_{\nu} \} $ denote the basis
 of the dual representation $V^*$ which is dual to the basis $\{ v_{\mu} \}$ of $V$; then $f_{\nu}$ is a weight vector of weight $-\nu$.  The representation
 $V^*$ is also minuscule, with highest weight vector $\{ f_{w_0 \gl} \}$ of weight $-w_0 \gl$.
  
Because $V$ occurs once in $R(\cm) = R(G/G^{\tilde{\nilp}})$, there is a unique (up to scaling)
$G$-invariant map $V \to R(G/G^{\tilde{\nilp}})$.  We can describe this map using Frobenius
 reciprocity.  Suppose $H$ is a subgroup of $G$.   If $V_2$ is a representation of $H$, then
 by definition, $\Ind_H^G V_2$ is the space of maps $r: G \to V_2$ satisfying
 $r(gh) = h^{-1} r(g)$ for all $g \in G$, $h \in H$, so if $V_2 = \C$ is the trivial representation,
 $\Ind_H^G V_2  = R(G/H)$.  If  $V_1$ is a representation of $G$, and $V_2$ a representation
 of $H$, Frobenius reprocity gives a bijection $\Hom_H(\Res^G_H V_1, V_2) \to \Hom_G(V_1, \Ind_H^G V_2)$.
 If this bijection takes
 $\varphi$ to $\Phi$, then the relation between $\varphi$ and $\Phi$ is as follows.
  If $v \in V_1$, $\Phi(v) \in \Ind_H^G V_2$ 
 satisfies 
 \begin{equation} \label{e.indrelation}
 \Phi(v)(g) = \varphi(g^{-1} v).
 \end{equation}  

In our situation, Frobenius reciprocity gives an isomorphism
\begin{equation} \label{e.ind}
\Hom_{G^{\nilp}}(\Res^G_{G^{\nilp}} V, \C) \cong \Hom_G(V, R(G/G^{\nilp})).
\end{equation}
Since
$V$ occurs once in $R(G/G^{\nilp})$,  the dimension of the right hand side of \eqref{e.ind} is $1$.
Hence $\dim \Hom_{G^{\tilde{\nilp}}}(V, \C) = \dim (V^*)^{ G^{\tilde{\nilp}}} = 1$.
The highest weight vector $f_{w_0 \gl}$ is in $(V^*)^{ G^{\tilde{\nilp}}} $, so it must span this space.
Let $\varphi = f_{w_0 \gl}$, and let $\Phi$ be the corresponding element
of $\Hom_G(V, R(G/G^{\nilp}))$.  By \eqref{e.indrelation},  $\Phi(v_{\mu}) (b) = f_{w_0 \gl}(b^{-1} v_{\mu})$.  Now,
 $f_{w_0 \gl}(v_{\mu})$ is $0$ unless $v_{\mu}$ is the lowest weight vector $v_{w_0 \gl}$ of $V$.  Also,
 since $b^{-1} \in B$, $b^{-1} v_{\mu}$ is a linear combination of $v_{\nu}$ with $\nu \geq \mu$.
 Hence, $\Phi(v_{\mu}) (b) = f_{w_0 \gl}(b^{-1} v_{\mu}) = 0$ unless $\mu = w_0 \gl$, and
 $\Phi(v_{w_0 \gl}) (b) = f_{w_0 \gl}(b^{-1} v_{w_0 \gl}) = e^{w_0 \gl}(b^{-1}) = e^{- w_0 \gl}(b)$.
This proves the claim.  The theorem follows.
\end{proof}

\begin{Cor}  $\overline{B \cdot \tilde{\nilp}}$ is normal if and only if for
all minimal dominant weights $\gl_{dom}$, when $\gl_{dom}$ is expressed as a sum
of simple roots, each coefficient is less than $1$.  In particular, if $G$ is simple,
$\overline{B \cdot \tilde{\nilp}}$ is normal if and only if $G$ is of type $A_1$ or $A_2$.
\end{Cor}

\begin{proof}
By Theorem \ref{thm.bclosure}, $\overline{B \cdot \tilde{\nilp}}$ is normal if and only if in each
coset of $\cp$ mod $\cq$, we have $\gl_{dom} = \gl_R$, which occurs exactly when
the expression for $\gl_{dom}$ as a sum
of simple roots has each coefficient less than $1$.  For simple $G$, this only occurs
when $G$ is of type $A_1$ or $A_2$, as can be seen from the expressions for the minimal dominant
weights in terms of simple roots given in \cite{Hum:72} (see Exercise 13 of Section 13.4, and Table 1 of
Section 13.2).
\end{proof}

\section{Other nilpotent orbits} \label{s.otherorbits}
In this section we extend the construction of $\widetilde{\cm}$ to covers of other
nilpotent orbits besides the principal nilpotent orbit.  
However, because we cannot apply the theory of toric varieties, the picture that is obtained is less complete than in
the case of the principal orbit.

The construction in the proof
of Theorem 4.1 of \cite{McG:89}, discussed in the introduction, is related to the
construction here, but it is somewhat different.  In \cite{McG:89}, the variety
$\widetilde{V}$ is the closure of a $P$-orbit in a representation of $G$, and the map
takes $G \times^P \widetilde{V}$ to the closure of the $G$-orbit in the representation.
Here, the variety $\widetilde{V}$ is different; the construction in this paper 
of $\widetilde{V}$ is modeled on the construction of $\widetilde{\fu}$
given in previous sections, and the map is to $\widetilde{\cm}_{\co} = \Spec R(\widetilde{\co})$
(which is the normalization of the closure of the $G$-orbit considered in \cite{McG:89}).

The meaning of the notation in this section is somewhat different than in other sections
because we are dealing with arbitrary nilpotent orbits.
We let $\nilp$ denote any nonzero nilpotent element of $\fg$.  
As in Section \ref{ss.prelimsemisimple}, choose a  standard $\mathfrak{sl}_2$-triple $\{h, e, f\}$ such that
$e = \nilp$ is the nilpositive element. Write $\fg_i$ for the $i$-eigenspace of $\ad h$ on $\g$; then
$\fg = \oplus \fg_i$.  Let $\fl = \fg_0$, $\fu_P = \fg_{\geq 1}$, and $\fp = \fl + \fu_P$; let $L$, $U_P$, and $P$
denote the corresponding subgroups of $G$.  Then $P$ is a parabolic subgroup of $G$ with Levi factor
$L$.   Let $V = \fg_{\geq 2} \subseteq \fu_P$.  If $\nilp$ 
an even nilpotent element (that is, if $\fg_i = 0$ for $i$ odd), as in the case when $\nilp$ is principal,
then $V = \fu_P$.  By \cite[Prop.~2.4]{BaVo:85}, we have $G^{\nilp} = P^v = L^{\nilp} U^{\nilp}$. The group
$L^{\nilp}$ is reductive; it is the centralizer in $G$ of the Lie subalgebra generated by the standard triple.
Since $U^{\nilp}$ is a unipotent group, it
is connected.  Write $C_{\nilp}$ for the component group $G^{\nilp}/G^{\nilp}_0 = L^{\nilp}/L^{\nilp}_0$.

Write $\co = G \cdot \nilp \cong G/G^{\nilp}$ for the orbit of $\nilp$ in $\fg$, and let $\cn_{\co} = \Spec R(\co)$.  Then
$\cn_{\co}$ is the normalization of $\overline{\co}$, the closure of $\co$ in $\fg$.  The map $\cn_{\co} \to
\overline{\co}$ is an isomorphism over $\co$, so we will view $\co$ as a subset of $\cn_{\co}$.
Write $\widetilde{\cn}_{\co}
= G \times^P V$.  The map
$G \times^P V \to \overline{\co}$ defined by $[g,v] \to g \cdot v$ is a resolution of singularities and an isomorphism
over $\co$ (see the proof of Theorem 3.1 in \cite{McG:89}).  This map factors through the normalization 
$\cn_{\co}$ of $\overline{\co}$, so we obtain maps
$$
\widetilde{\cn}_{\co} \stackrel{\mu}{\longrightarrow} \cn_{\co} \longrightarrow \overline{\co}.
$$

The universal cover of $\co$ is $\widetilde{\co} = G/G^{\nilp}_0$.  Write $\widetilde{\nilp}$ for the
identity coset $1 \cdot G^{\nilp}_0$ in $\widetilde{\co}$; then $G^{\widetilde{\nilp}}= P^{\widetilde{\nilp}} = G^{\nilp}_0$.
We write $\cm_{\co} = \Spec R(\widetilde{\co})$.  Since $\widetilde{\co}$ is normal (in fact, smooth),
$\cm_{\co}$ is a normal variety.  The variety $\widetilde{\co}$ is quasi-affine (see the proof of
Theorem 4.1 in \cite{McG:89}).  Lemma \ref{lem.normalization} and Corollary \ref{cor.normalization} extend to this situation, with
the same proof, so we obtain a commutative diagram
$$
\begin{CD}
\widetilde{\co} @>>> \cm_{\co} \\
@VVV     @VVV \\
\co @>>> \cn_{\co},
\end{CD}
$$
where the horizontal maps are open embeddings, and the vertical maps are quotients by $C_{\nilp}$.
Moreover, the restriction maps $R(\cn_{\co}) \to R(\co)$ and $R(\cm_{\co}) \to R(\widetilde{\co})$ are
isomorphism.

Because $L^{\nilp}$ and $L^{\nilp}_0$ are reductive, the varieties $ L/L^{\nilp}$ and $ L/L^{\nilp}_0$ are affine, so
their function fields are the fields of fractions of $R( L/L^{\nilp})$ and $R( L/L^{\nilp}_0)$, respectively.
We define $\toric_{ad} = \fg_2$.  Kostant (see \cite[Section 4.2]{Kos:59}) proved that the orbit $L \cdot \nilp \cong L/L^{\nilp}$ is open
in $\toric_{ad}$.  Kostant observed \cite{Kos:90} that the complement $L \cdot \nilp$ in
$\toric_{ad}$ is the zero set of a single function $f$; this follows because 
 $L/L^{\nilp}$ is an affine subvariety of $\toric_{ad}$.  (In fact, Kostant proved more about $f$, but
 we do not need this here.)
We conclude that $R(L/L^{\nilp}) = R(\toric_{ad})[\frac{1}{f}]$.  

Kostant also proved that $P \cdot \nilp = L \cdot \nilp + \fg_{\geq 3}$ (\cite[Section 4.2]{Kos:59}).  Hence $P \cdot \nilp$ is open
in $V$.
Let $A$ denote the integral closure of $R(\toric_{ad})$ in the function field $\kappa(L/L^{\nilp}_0)$
and let $\toric = \Spec R(A)$, so $R(\toric) = A$.  
The ring $R(\toric)$ is stable under the action of $C_{\nilp}$, and $R(\toric)^{C_{\nilp}} = R(\toric_{ad})$ (see 
\cite[Ch.~5, Ex.~14]{AtMa:69}).  Hence $C_{\nilp}$ acts on $\toric$, and $\toric/C_{\nilp} = \toric_{ad}$.  

\begin{Lem} \label{lem.open}
We have an $L$-equivariant commutative
diagram
\begin{equation} \label{e.opendiagram}
\begin{CD}
L/L^{\nilp}_0 @>>> \toric \\
@VVV     @VVV \\
L/L^{\nilp} @>>> \toric_{ad},
\end{CD}
\end{equation}
where the horizontal maps are open embeddings, and the vertical maps are quotients by $C_{\nilp}$.
\end{Lem}

\begin{proof}
Let $K$ denote the function field $\kappa(L/L^{\nilp}_0)$.
The ring $R(L/L^{\nilp}_0)$ is integrally closed in $K$, since $L/L^{\nilp}_0$ is smooth, hence
normal.  Also, $R(L/L^{\nilp}_0)$ is integral over $R(L/L^{\nilp})$, since $R(L/L^{\nilp}_0)^{C_{\nilp}} = R(L/L^{\nilp})$.
Hence $R(L/L^{\nilp}_0)$ is the integral closure in $K$ of $R(L/L^{\nilp})$.  On the other hand,
$R(L/L^{\nilp}) = R(\toric_{ad})[\frac{1}{f}]$.  By definition, $R(\toric)$ is the integral closure of $R(\toric_{ad})$ in
$K$.  The compatibility of integral closure with localization (\cite[Prop.~5.12]{AtMa:69}) implies
that the integral closure of $R(L/L^{\nilp})$ in $K$ is $R(\toric)[\frac{1}{f}]$.  Hence
$R(L/L^{\nilp}_0) = R(\toric)[\frac{1}{f}]$, so $L/L^{\nilp}_0$ is an open subvariety of $\toric$.
The commutative diagram \eqref{e.opendiagram} follows from considering the corresponding
diagram of rings of regular functions (the arrows are reversed).  In the diagram \eqref{e.opendiagram},
the left vertical and bottom horizontal maps are $L$-equivariant by construction.  The right vertical
map is $L$-equivariant because it corresponds to the inclusion $R(\toric_{ad}) \to R(\toric)$, and
$R(\toric)$ is stable under $L$ since it is the integral closure of the ring $R(\toric_{ad})$, on which
$L$ acts.  The top horizontal arrow is $L$-equivariant since it corresponds to the inclusion
$R(\toric) \hookrightarrow R(L/L^{\nilp}_0)$, and both $R(\toric)$ and $R(L/L^{\nilp}_0)$ are stable
under $L$.
\end{proof}

We will abuse notation and write $1$ for the image of the identity coset $1 \cdot L/L^{\nilp}_0$ in $\toric$.

We now define a variety $\widetilde{V}$ which is analogous to the variety $\widetilde{\fu}$ defined
when we studied the principal nilpotent orbit.
Recall that $V = \fg_{\geq 2}$.  We can identify $\toric_{ad} = \fg_2 \cong V/\fg_{\geq 3}$; with this
identification, $\toric_{ad}$ has an action of $P$, where the subgroup $U_P$ acts trivially.  
We extend the $L$-action on $\toric$ to a $P$-action by requiring that $U_P$ act trivially; then
the map $\toric \to \toric_{ad}$ is $P$-equivariant.  The
projection $V \to \toric_{ad}$ is $P$-equivariant.  We define
$$
\widetilde{V} = \toric \times_{\toric_{ad}} V.
$$
Because both factors in the fiber product have $P$-actions, and the maps are $P$-equivariant,
the variety $\widetilde{V}$ has a $P$-action.  We have $P^{(1,\nilp)} = P^{\widetilde{\nilp}} = P^{\nilp}_0$.
Since the orbit $P \cdot \nilp$ is open in $V$, the orbit $P \cdot (1,\nilp) \cong P \cdot \widetilde{\nilp}$ is
open in $\widetilde{V}$.  The analogue of Proposition \ref{p.normalization} holds in this setting.

\begin{Prop}  \label{p.normalization2}
The map $P \cdot (1, \nilp) \to \cm_{\co} $ extends to a map $\phi: \tilde{V} \rightarrow \cm_{\co}$.
We have a commutative diagram
\begin{equation} \label{e.normvar2}
 \begin{CD}
 \tilde{V} @>>> \cm_{\co} \\
 @VVV    @VVV \\
V @>>> \cn_{\co}.
 \end{CD}
\end{equation}
Moreover, the resulting map
 $\psi: \tilde{V} \to \overline{ P \cdot \tilde{\nilp} }$
 is the normalization map.  We have $\psi^{-1}(P \cdot \tilde{\nilp}) = P \cdot (1,\nilp)$,
 and $\psi$ restricts to an isomorphism $P \cdot (1,\nilp) \to P \cdot \tilde{\nilp}$.
\end{Prop}

The proof is essentially the same as the proof of Proposition \ref{p.normalization}; we omit the details.

Define $\tilde{\cm_{\co}} = G \times^P \tilde{V}$.  There is a $G$-equivariant map $\tilde{\mu}: \tilde{\cm_{\co}} \to \cm_{\co}$ defined
by $\tilde{\mu}([g,\xi]) = g \cdot \phi(\xi)$.
Using Proposition \ref{p.normalization2}, we obtain the analogue of Theorem \ref{t.genspringer} for an arbitrary
nilpotent orbit.

\begin{Thm} \label{t.genspringer2}
Let $\co = G \cdot \nilp$ be a nonzero nilpotent orbit in $\fg$.
The map $\tilde{\mu}: \tilde{\cm_{\co}} \to \cm_{\co}$ is proper and is an isomorphism over $\tilde{\co}$.  There is a commutative
diagram
$$
\begin{CD} \label{e.commgenspringer2}
\tilde{\cm_{\co}} @>{\mnmap}>> \tilde{\cn_{\co}} \\
@V{\mmap}VV         @VV{\mu}V \\
\cm_{\co} @>{\eta}>> \cn_{\co}.
\end{CD}
$$
The horizontal maps in this diagram are quotients by $C_{\nilp}$.
\end{Thm}

We omit the proof, which is similar to the proof of Theorem \ref{t.genspringer}.

The variety $\tilde{\cm}_{\co}$ is normal, but for a general nilpotent orbit we do not know much else about it, since we do not know much about the variety $\toric$.
 However, $\toric$ can be viewed as a partial compactification of
the homogeneous variety $L/L^{\nilp}_0$, and perhaps the Luna-Vust theory of such compactifications
can be applied to obtain more information.

\bibliographystyle{amsalpha}
\bibliography{niltoric}

\def\cprime{$'$}
\providecommand{\bysame}{\leavevmode\hbox to3em{\hrulefill}\thinspace}
\providecommand{\MR}{\relax\ifhmode\unskip\space\fi MR }
\providecommand{\MRhref}[2]{%
  \href{http://www.ams.org/mathscinet-getitem?mr=#1}{#2}
}
\providecommand{\href}[2]{#2}
\begin{thebibliography}{GPR19}

\bibitem[AM]{AtMa:69}
M.~F. Atiyah and I.~G. Macdonald, \emph{Introduction to commutative algebra},
  economy ed., Addison-Wesley Series in Mathematics.

\bibitem[Bro98]{Bro:98}
Abraham Broer, \emph{Decomposition varieties in semisimple {L}ie algebras},
  Canad. J. Math. \textbf{50} (1998), no.~5, 929--971. \MR{1650954}

\bibitem[BV85]{BaVo:85}
Dan Barbasch and David~A. Vogan, Jr., \emph{Unipotent representations of
  complex semisimple groups}, Ann. of Math. (2) \textbf{121} (1985), no.~1,
  41--110. \MR{782556}

\bibitem[Che58]{Che:58}
\emph{S\'{e}minaire {C}. {C}hevalley; 2e ann\'{e}e: 1958. {A}nneaux de {C}how
  et applications}, Secr\'{e}tariat math\'{e}matique, 11 rue Pierre Curie,
  Paris, 1958. \MR{0110704}

\bibitem[Ful93]{Ful:93}
William Fulton, \emph{Introduction to toric varieties}, Annals of Mathematics
  Studies, vol. 131, Princeton University Press, Princeton, NJ, 1993, The
  William H. Roever Lectures in Geometry. \MR{1234037}

\bibitem[GPR19]{GrPrRu:19}
William Graham, Martha Precup, and Amber Russell, \emph{A new approach to the
  generalized {S}pringer correspondence}, in preparation, 2019.

\bibitem[Gra92]{Gra:92}
William~A. Graham, \emph{Functions on the universal cover of the principal
  nilpotent orbit}, Invent. Math. \textbf{108} (1992), no.~1, 15--27.
  \MR{1156383}

\bibitem[Har77]{Har:77}
Robin Hartshorne, \emph{Algebraic geometry}, Springer-Verlag, New
  York-Heidelberg, 1977, Graduate Texts in Mathematics, No. 52. \MR{0463157}

\bibitem[Hes76]{Hes:76}
Wim~H. Hesselink, \emph{Cohomology and the resolution of the nilpotent
  variety}, Math. Ann. \textbf{223} (1976), no.~3, 249--252. \MR{417195}

\bibitem[Hin91]{Hin:91}
V.~Hinich, \emph{On the singularities of nilpotent orbits}, Israel J. Math.
  \textbf{73} (1991), no.~3, 297--308. \MR{1135219}

\bibitem[Hum72]{Hum:72}
James~E. Humphreys, \emph{Introduction to {L}ie algebras and representation
  theory}, Springer-Verlag, New York, 1972, Graduate Texts in Mathematics, Vol.
  9.

\bibitem[Jan04]{Jan:04}
Jens~Carsten Jantzen, \emph{Nilpotent orbits in representation theory}, Lie
  theory, Progr. Math., vol. 228, Birkh\"{a}user Boston, Boston, MA, 2004,
  pp.~1--211. \MR{2042689}

\bibitem[Kos59]{Kos:59}
Bertram Kostant, \emph{The principal three-dimensional subgroup and the {B}etti
  numbers of a complex simple {L}ie group}, Amer. J. Math. \textbf{81} (1959),
  973--1032. \MR{114875}

\bibitem[Kos63]{Kos:63}
\bysame, \emph{Lie group representations on polynomial rings}, Amer. J. Math.
  \textbf{85} (1963), 327--404. \MR{158024}

\bibitem[Kos90]{Kos:90}
\bysame, {Private Communication}, 1990.

\bibitem[McG89]{McG:89}
William~M. McGovern, \emph{Rings of regular functions on nilpotent orbits and
  their covers}, Invent. Math. \textbf{97} (1989), no.~1, 209--217. \MR{999319}

\bibitem[Rus12]{Rus:12}
Amber Russell, \emph{Graham's variety and perverse sheaves on the nilpotent
  cone}, Ph.D. thesis, Louisiana State University, 2012.

\bibitem[Tho87]{Tho:87}
R.~W. Thomason, \emph{Algebraic {$K$}-theory of group scheme actions},
  Algebraic topology and algebraic {$K$}-theory ({P}rinceton, {N}.{J}., 1983),
  Ann. of Math. Stud., vol. 113, Princeton Univ. Press, Princeton, NJ, 1987,
  pp.~539--563.

\end{thebibliography}

\end{document}